\documentclass[reqno, 12pt]{amsart}
\usepackage{mathrsfs}

\let\a\alpha    \let\c\gamma  
   \let\e\epsilon
\let\l\lambda  \let\k\chi 
  \let\GD\Delta    
\def\GL2{{\rm GL}_2}
\def\SL2{{\rm SL}_2}
\def\C{\mathbf C}

\def\A{\mathbf A}

\def\bk{\mathbf k}

\def\U{\mathcal U}
\def\F{{\mathbb F}}
  \def\Fq{{\F_{q}}}
\def\Z{{\mathbb Z}}

\def\P{\mathbb P}
\def\P1{{\mathbb P}^1}

\def\z{\rm z}

\def\bv{\mathbf v}
\def\har{\rm har}

\def\la{\langle}
\def\ra{\rangle}

\newtheorem{lemma}{Lemma}[section]
\newtheorem{cor}{Corollary}[section]
\newtheorem{prop}{Proposition}[section]

\newtheorem{theorem}{Theorem}[section]

\theoremstyle{definition}
\newtheorem{defn}{Definition}[section]

\theoremstyle{remark}
\newtheorem{rem}{Remark}[section]

\newtheorem{example}{Example}[section]

\numberwithin{equation}{section}

\usepackage[mathscr]{eucal}
\usepackage[all]{xy}
\usepackage{graphicx}

\begin{document}

\title[characteristic $p$ valued measure]{On the characteristic $p$ valued measure associated to
Drinfeld discriminant}
\author{Zifeng Yang}
\address{Department of Mathematics \\Capital Normal University\\
Beijing, 100037\\ P. R. China} \email{yangzf@mail.cnu.edu.cn}
\date{}

\thanks{The research of the author is partially supported by
``NSFC 10871107" and a research program in mathematical automata ``KLMM0914".}

\begin{abstract}
In this paper, after reviewing known results on functions over Bruhat-Tits
trees and the theory of characteristic $p$ valued modular forms,
we present some structure of the tempered distributions on
the projective space ${\P1}(\bk_\infty)$ over a complete function field $\bk_\infty$
of characteristic $p$,
and calculate the characteristic $p$ valued measure associated to the Drinfeld
discriminant and the characteristic $p$ valued measure associated to the
Poinca\'e series.
\end{abstract}

\maketitle


\vspace{24pt}

Throughout this paper, $q$ is a fixed power of a
prime $p$, $\Fq$ is the finite field of $q$ elements. Let $\A={\Fq}[T]$ be the
polynomial ring with the discrete valuation at $\infty$, and $\bk$ be the fraction field of $\A$.
For an element $\a$ in $\A$,
we denote by $\deg_T(\a)$ the degree of $\a$ as a polynomial in $T$, the
subscript $T$ is omitted if it does not cause any confusion. By
convention, $\deg(0)=-\infty$.
On the completion $\bk_\infty=\Fq((\frac{1}{T}))$ of $\bk$ at $\infty$, we take $\pi=\frac{1}{T}$ as the
uniformizer, and set ${\A}_{\infty}=\Fq[[\pi]]=\Fq[[\frac{1}{T}]]$. The valuation $v$ and
absolute value $|\cdot|$ of $\bk_\infty$ are normalized such that
$v(\pi)=1$, and $|\pi|=\frac{1}{q}$. Let ${\C}_\infty$ be the completion of an algebraic closure of
${\bk}_\infty$ with the absolute value also denoted by $|\cdot|$.
And we set $\Gamma={\GL2}(\A)$.

%\setcounter{section}{-1}
%================================================================================================================
\section{Introduction}\label{introduction}
Let ${\A}^{+}=\{\text{monic polynomials in } \Fq[T]\}$. Over the function field $\bk$, the analogue of
the Riemann-zeta values \cite{Ca1935} is given as
\[
{\z}(n)=\sum\limits_{a\in {\A}^{+}}\frac{1}{a^n}
\]
for $n\in {\Z}^{+}$. The analogous zeta function is given by Goss \cite[Chapter 8]{Go1998}
as an analytic function $\zeta(s)$ over the ``complex plane" $S_{\infty}={\C}_{\infty}^{\ast}
\times {\Z}_p$ with $s=(x,y)\in S_{\infty}$:
\begin{equation}\label{e:1.1}
\zeta(s)=\sum\limits_{a\in {\A}^{+}}\frac{1}{a^s}=\sum\limits_{j=0}^\infty x^{-j}\left(
\sum\limits_{\substack{a\in {\A}^{+}\\\deg(a)=j}} \la a\ra^{-y}\right),
\end{equation}
where $\la a\ra=T^{-\deg(a)}\cdot a$ is the $1$-unit part of $a\in {\bk}_{\infty}^{\times}$,
and $a^s = x^{\deg(a)}\cdot (\la a\ra)^y$. Under this construction, we get
the power sum ${\z}(n)=\zeta(T^n,n)$ for an integer $n\ge 1$. Goss \cite{Go1998} also constructs
the general $L$-series over the ``complex plane" $S_{\infty}$.

One aspect of the analytic theory over function fields is the study of
the Drinfeld's upper half plane $\Omega={\C}_{\infty}-{\bk}_{\infty}$, which
comes as the rank two case of the analytic structures related to the moduli spaces
of Drinfeld modules, see \cite{Ge1986} for an exposition. There are
rich structures of rigid analytic spaces on ${\Omega}$, so that we can
talk about the concepts like modular forms and Hecke operators over function fields,
see \cite{Go1980a}~\cite{Go1980b}~\cite{Go1980c}~\cite{Ge1986}~\cite{Ge1988}~\cite{Ge1996}, etc.
But the Hecke operators thus defined so far can't distinguish the modular forms,
and we don't know if there is any analogue of Mellin transforms which would
relate the modular forms to the $L$-series over $S_\infty$(in Goss' sense),
even for the
$L$-series coming from geometry as \cite{Bo2005} points out.
Since there are no Haar measures of characteristic $p$, a general idea of dealing
with this problem is to study the measures (more precisely,
distributions) which define the $L$-series through integrations over power functions,
and the measures associated to modular forms.

On the measures associated to the Goss zeta function $\zeta(s)$ in equation (\ref{e:1.1}),
we have the computations by Thakur \cite{Th1990} at finite places and by Yang \cite{Ya2001}
at the place $\infty$. The special values $\zeta(x,-j)$ are determined by the formula \cite{Ya2001}
\begin{equation}\label{e:1.2}
\zeta(x,-j)=\int_{{\A}_{\infty}} t^j d\mu_x^{(\infty)}(t)
=\int_{{\U}_1} t^j d(x\nu_x^{(\infty)})(t)
\end{equation}
for integers $j\ge 0$, where ${\U}_1$ is the $1$-units of ${\A}_{\infty}$, and
$\mu_x^{(\infty)}$ and $\nu_x^{(\infty)}$
are measures on ${\A}_{\infty}$ depending on the variable $x$.

On the measures associated to the modular forms, Teitelbaum \cite{Te1991} established an
isomorphism between cusp forms of an arithmetic group $G$ of ${\GL2}(\bk)$ and
the $G$-equivariant harmonic co-cycles on the Bruhat-Tits tree ${\mathcal T}$
associated to $\bk_{\infty}$. A harmonic co-cycle on ${\mathcal T}$ can also be
viewed as a measure on ${\P1}(\bk_{\infty})$, therefore Teitelbaum's theorem gives
a way to compare the relevant measures we have talked above. Goss \cite{Go1992} has pointed
out some implications from Teitelbaum's construction of measures. But the construction of the
harmonic co-cycles in Teitelbaum's theorem is quite abstract, it is not easy to understand
what is behind these measures. We will carry out an explicit computation of the
measure associated to the Drinfeld discriminant in Teitelbaum's theorem, wish to
understand these topics better.

We summarize some well-known facts about Bruhat-Tits trees in Section \ref{btt},
and give a simple introduction to characteristic $p$ valued modular forms in
Section \ref{mf}.

In Section \ref{cpvm}, we talk about the
the space of $C^h$ functions on an open compact subset $S$ of ${\P1}(\bk_\infty)$
and study its dual space, which is proved to be the space of $h$-admissible measures
on $S$.

In Section \ref{fbtt}, we introduce some known results of the theory of functions on Bruhat-Tits trees, and Teitelbaum's theorems on the correspondence
between cusp forms of an arithmetic subgroup $G$ of ${\GL2}(\bk)$ and $G$-equivariant
harmonic cocycles on Bruhat-Tits trees.

In Section \ref{mdd}, we explicitly determine the characteristic $p$ valued measure
associated to the Drinfeld discriminant $\Delta(z)$.

In Section \ref{comp}, we calculate some special values of the $L$-function
associated to the Drinfeld discriminant, and also give similar results on the
characteristic $p$ valued measure associated to the Poinca\'e series. Some comments
about characteristic $p$ valued modular forms and harmonic functions on Bruhat-Tits
trees are made, and some possible application to algebraic ergodic theory is also
mentioned in this section.

%================================================================================================================

\section{Bruhat-Tits trees}\label{btt}

Let $V={\bk}_\infty\oplus {\bk}_\infty$ be a two dimensional vector space over ${\bk}_\infty$. A lattice $L$ of $V$
is a finitely generated free ${\A}_\infty$ module of rank $2$. Two lattices $L$ and $L'$ are said to
be equivalent if there exists a $\lambda\in {\bk}_\infty^*$ such that $L'=\lambda L$. Let $V_{\mathcal T}$
denote the set of all equivalence classes $[L]$ of lattices $L$ of $V$. For any elements
$\Lambda, \Lambda'\in V_{\mathcal T}$, we can find lattices $L\in \Lambda$ and $L'\in \Lambda'$
such that $L'\subset L$ and $L/L'\cong {\A}_\infty/{\pi}^n{\A}_\infty \cong {\A}/{\pi}^n{\A}$
for some non-negative integer $n$, which we denote by $d({\Lambda},{\Lambda}')=n$.
We define
\[
E_{\mathcal T}=\{e_{\Lambda\Lambda'}: \,\, d({\Lambda},{\Lambda}')=1\}.
\]
Then the Bruhat-Tits tree $\mathcal T$ consists of the set $V_{\mathcal T}$ of vertices and the
set $E_{\mathcal T}$ of (oriented, with orientation to be determined later in this section) edges.
The action of ${\GL2}(\bk_\infty)$ on $V$ induces an action
of ${\GL2}(\bk_\infty)$ on $\mathcal T$ (in fact, an action of ${\rm PGL}_2(\bk_\infty)$ on $\mathcal T$).

Suppose that $G$ is a group acting on a graph $X$. Then we have the quotient graph $G\backslash X$. A
fundamental domain of $G\backslash X$ is a subgraph $Z$ of $X$ such that
the natural map $Z\to G\backslash X$ is an
isomorphism of graphs. A fundamental domain of a graph by a group action doesn't necessarily exist, but
we have results in some special cases.
Let
\[
{\GL2}(\bk_\infty)^+=\{\alpha\in {\GL2}(\bk_\infty):\,\, v(\det \alpha)\text{ is an even integer}\}.
\]
Let ${\Lambda}_n$ denote the class of the lattice ${\A}_\infty \oplus \pi^n{\A}_\infty
\sim T^n{\A}_\infty \oplus {\A}_\infty$ for $n\in {\Z}$.
We know the following properties of the actions of subgroups of ${\GL2}(\bk_\infty)$ on
$\mathcal T$, the detail can be found in Serre's book \cite{Se1980}.
\begin{itemize}
\item Let $G$ be a subgroup of ${\GL2}(\bk_\infty)^+$. If the closure of $G$ in ${\GL2}(\bk_\infty)$ contains
${\SL2}(\bk_\infty)$, then the fundamental domain of the action of $G$ on $\mathcal T$
is
\[
\xymatrix@R=0pt{{\circ}\ar@{-}[r]&{\circ}\\ \Lambda_0 & \Lambda_1}.
\]
\item The fundamental domain of the action of $\Gamma={\GL2}(\A)$
 on $\mathcal T$ is
\begin{equation}\label{e:2.1}
\xymatrix@R=0pt{{\circ}\ar@{-}[r]&{\circ}\ar@{-}[r]&{\circ}\ar@{.}[r] &{\circ}\ar@{-}[r]&\ar@{.}[r]&\\
\Lambda_0 & \Lambda_1 & \Lambda_2 & \Lambda_n &  }.
\end{equation}
\end{itemize}
The action of ${\GL2}(\bk_\infty)$ on $V_{\mathcal T}$ is transitive,
thus we have the bijections
\begin{equation}\label{e:2.2}
V_{\mathcal T}\mathop{\longleftrightarrow}\limits^{\cong}
{\GL2}(\bk_\infty)/{\rm Stab}({\Lambda_0})
 = {\GL2}(\bk_\infty)/({\bk^*_\infty}\cdot {\GL2}(\A_\infty)),
\end{equation}
where ${\rm Stab}({\Lambda_0})$ denotes the stabilizer of
${\Lambda}_0$ in ${\GL2}({\bk}_\infty)$.
Therefore we can use matrices to express the vertices of $\mathcal T$ in terms of the
coset representatives in the above relation:
\begin{equation}\label{e:2.3}
V_{\mathcal T}=\left\{ \left(\begin{array}{cc}\pi^k & u\\0 & 1\end{array}\right): \, \,
k\in {\mathbb Z}, u\in{\bk}_\infty, u \,{\rm mod}\, \pi^k{\A}_\infty\right\}.
\end{equation}
Let ${\bv}_1=(1,0)^T, {\bv}_2=(0,1)^T$ be the standard basis of
$V={\bk}_\infty\oplus {\bk}_\infty$. Under
the one to one correspondence (\ref{e:2.2}) and the representations
of the coset representatives (\ref{e:2.3}),
we have
\begin{equation}\label{e:2.4}
[\pi^k{\bv}_1,u{\bv}_1+{\bv}_2]\longleftrightarrow
\left(\begin{array}{cc}\pi^k & u\\ 0 & 1\end{array}\right),
\end{equation}
where $k$ and $u$ are as in (\ref{e:2.3}), and $[\pi^k{\bv}_1,u{\bv}_1+{\bv}_2]$ denotes the equivalence class
of the ${\A}_\infty$-lattice generated
by $\pi^k{\bv}_1$ and $u{\bv}_1+{\bv}_2$.

An end of the tree $\mathcal T$ is an infinite path starting at some vertex and without back-tracking,
for example, the half line in (\ref{e:2.1}) from the vertex ${\Lambda}_0$ to
${\Lambda}_1$, $\Lambda_2$, and so on. Two ends are equivalent if
and only of they differ by a finite number of vertices and edges.
The set of ends of $\mathcal T$ is in bijection with ${\mathbb
P}^1(\bk_\infty)$. This bijection is set up in this paper as
follows. We choose the vertex $\Lambda_0$ as the starting point
of any end (in some equivalence class), then $\infty\in {\mathbb
P}^1(\bk_\infty)$ corresponds to the end (\ref{e:2.1}). For
$x=\sum_{n=n_0}^\infty c_n \pi^n \in {\bk}_\infty$, where $c_n\in
{\Fq}$, the corresponding end (which is in some equivalence class; but the end of
the following graph may not start with the vertex $\Lambda_0$) is
\begin{equation}\label{e:2.4.5}
\xymatrix@R=0pt{{\circ}\ar@{-}[r]&{\circ}\ar@{-}[r]&{\circ}\ar@{.}[r] &{\circ}\ar@{-}[r]&\ar@{.}[r]&\\
[{\bv}_1,x\cdot {\bv}_1+{\bv}_2] & [\pi{\bv}_1,x\cdot {\bv}_1+{\bv}_2] &
 & [\pi^n{\bv}_1,x\cdot {\bv}_1+{\bv}_2] &  }
\end{equation}
where we notice that $[\pi^n{\bv}_1,x\cdot {\bv}_1+{\bv}_2]$
$=[\pi^n{\bv}_1,x_{[n-1]}\cdot {\bv}_1+{\bv}_2]$, with
$x_{[n]}=\sum_{k=n_0}^{n}c_k \pi^k$.

The edges of $\mathcal T$ are oriented, with the set $E^{+}_{\mathcal T}$ of edges of positive orientation
given in terms of the $\infty$ end of $\mathcal T$:
\begin{equation}\label{e:2.5}
\xymatrix@R=0pt{{\circ}\ar@{->}[r]&{\circ}\ar@{->}[r]&{\circ}\ar@{.>}[r] &{\circ}\ar@{->}[r]&\ar@{.}[r]&\\
\Lambda_0 & \Lambda_1 & \Lambda_2 & \Lambda_n &  },
\end{equation}
any edge with a consistent orientation with the above is in $E^+_{\mathcal T}$ ( since $\mathcal T$ is
connected ).

Let $\Gamma_0={\GL2}(\Fq)$, a subgroup of $\Gamma={\GL2}(\A)$, and
\[
\Gamma_n=\left\{\left(\begin{array}{cc}a & b\\0 & d\end{array}\right): \,\,
a,d\in {\F}^*_q, b\in {\Fq}[T], \deg_T(b)\le n\right\},
\]
for $n\ge 1$. Then we have
\begin{prop}[Serre, \cite{Se1980}]\label{prop:2.1}\begin{itemize}\item[(1)] $\Gamma_n$ is the stabilizer of
$\Lambda_n$.
\item[(2)] $\Gamma_0$ acts transitively on the set of edges with origin $\Lambda_0$.
\item[(3)] For $n\ge 1$, $\Gamma_n$ fixes the edge $\Lambda_n\Lambda_{n+1}$
and acts transitively on the
set of edges with origin $\Lambda_n$ but distinct from
the edge $\Lambda_n\Lambda_{n+1}$.
\end{itemize}
\end{prop}

%===========================================================================================================

\section{Modular forms}\label{mf}

The group ${\GL2}({\bk}_\infty)$ acts on the Drinfeld's upper half
plane $\Omega={\C}_\infty-\bk_{\infty}
={\mathbb P}^1({\C}_\infty)-{\mathbb P}^1({\bk}_\infty)$
by
\[
\gamma\cdot z=\frac{az+b}{cz+d}
\]
for $\gamma=\left(\begin{array}{cc}a&b\\ c&d\end{array}\right)
\in {\GL2}({\bk}_\infty)$ and $z\in \Omega$.
This action is also written as $\gamma(z)$ or $\gamma z$.
The space $\Omega$ has a good covering $\{D_i\}_{i\in I}$,
where $I=\{(n,x): \,\, n\in {\Z}, x\in
{\bk}_\infty/\pi^{n+1} {\A}_\infty\}$, $D_{(n,x)}=x+D_n$, and
\[
D_n=\left\{z\in {\C}_\infty:\,\,
\begin{array}{ll}
&|\pi^{n+1}|\le |z|\le |\pi^n|,\\
&|z-c\pi^n|\ge |\pi^n|, |z-c\pi^{n+1}|\ge |\pi^{n+1}| \text{ for all $c\in {\F}_q^*$ }
\end{array}
\right\}.
\]
These subsets $D_i$'s of $\Omega$ are affinoid spaces and $\Omega$ has a rigid analytic structure, see
\cite{Ge1996} and \cite{GP1980} for the details. Hence we can talk about the rigid analytic functions
on $\Omega$ and study their properties.

A finitely generated $\A$-submodule $L\subset {\C}_\infty$ is said to be
an $\A$-lattice of ${\C}_\infty$
if $L$ is discrete in the topology of ${\C}_\infty$. Hence an $\A$-lattice
$L$ of ${\C}_\infty$ is a projective
$\A$-module of finite rank $d$, and we have an isomorphism of $\A$-modules
\[
L\cong {\A}^d
\]
since ${\A}={\Fq}[T]$ is a principal ideal domain.

\begin{example}\label{example:3.1}
Let $L\subset {\C}_\infty$ be an $\A$-lattice of rank $1$, in particular, a fractional ideal of $\A$. Then
\[
e_L(z)=z\prod\limits_{0\ne a\in L}\left(1-\frac{z}{a}\right)
\]
and
\begin{equation}\label{e:3.1}
t_L(z)=e^{-1}_L(z)=\sum\limits_{a\in L}\frac{1}{z-a}
\end{equation}
are $\Fq$-linear rigid analytic functions on $\Omega$. These two
functions are invariant under $L$-translations.
See \cite{Go1998} etc.
\end{example}

Generally, a subgroup $G$ of ${\GL2}({\bk})$ is called arithmetic
if $G$ is contained in ${\rm GL}(Y)$ and contains the kernel $G(Y,{\mathfrak a})$
of the reduction map
${\rm GL}(Y)\to {\rm GL}(Y/{\mathfrak a}Y)$ for some rank two projective $\A$ submodule
$Y$ of ${\bk}\oplus {\bk}$ and some ideal $\mathfrak a$ of $\A$.
In this paper, we assume that $\A$ is a polynomial ring over $\Fq$,
hence the rank $2$ projective
$\A$-module $Y$ is free: $Y \cong {\A}\oplus {\A}$, thus we
only consider the arithmetic
subgroups of ${\GL2}(\A)$.

As a subgroup of ${\GL2}(\bk_\infty)$, an arithmetic subgroup
$G$ acts on the Bruhat-Tits tree $\mathcal T$. The quotient
$G\backslash {\mathcal T}$ is a finite graph joined by
a finite number of ends. For example, in the following graph, there are two
ends on the left and on the right, and
the middle is finite.
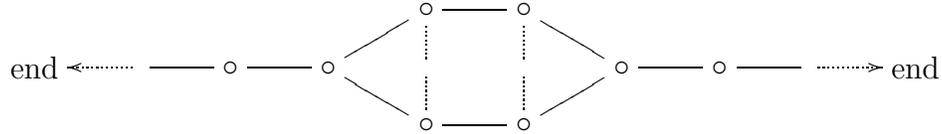
\begin{equation*}
\xymatrix@R=8pt{
          &          &                 &                  &{\circ}\ar@{-}[r]\ar@{.}[d]&{\circ}\ar@{-}[dr]\ar@{.}[d] &  &\\
{\rm end}&\ar@{-}[r]\ar@{.>}[l]&{\circ}\ar@{-}[r]&{\circ}\ar@{-}[ur]\ar@{-}[dr]   &
 &   &{\circ}\ar@{-}[r]&{\circ}\ar@{-}[r]&\ar@{.>}[r]&{\rm end}\\
          &         &                 &                  &{\circ}\ar@{-}[r]\ar@{.}[u]&{\circ}\ar@{-}[ur]\ar@{.}[u] &}
\end{equation*}
These ends are called the cusps of the arithmetic subgroup $G$. We have the bijection
\[
\{\text{cusps of $G$}\}\mathop{\longrightarrow}\limits^{\cong}
G\backslash {\mathbb P}^1(\bk).
\]
We refer \textsection $2$, Chapter II of \cite{Se1980} for the details. Thus we see
from (\ref{e:2.1}) that
the group $\Gamma$ has exactly one cusp $\infty$.

\begin{defn}\label{defn:3.1}
Let $G$ be an arithmetic subgroup of ${\GL2}({\bk})$. A rigid analytic
function $f:\Omega\to {\C}_\infty$ is
called a modular form with respect to the arithmetic subgroup $G$ of weight $k$
and type $m$
for a non-negative integer $k$ and a class $m$ in ${\Z}/(q-1)$ if the following
two conditions are
satisfied:
\begin{itemize}
\item[(1)] for any $\gamma=\left(\begin{array}{cc}a&b\\ c&d\end{array}\right)
\in G$, the
following equation holds:
\[
f|{[\gamma]_{k,m}}(z)=f(z),
\]
where $f|{[\gamma]_{k,m}}(z):= (\det(\gamma))^{m}j(\gamma,z)^{-k}
f(\gamma z)$, and
\begin{equation}\label{e:3.1.5}
j(\gamma,z)=cz+d;
\end{equation}
\item[(2)] $f$ is holomorphic at every cusp of $G$.
\end{itemize}
\end{defn}
In the above definition, the second condition is interpreted as follows. For a cusp $p$
of $G$, we take an element $\rho
\in {\GL2}(\bk)$ such that $\rho(\infty)=p$. Then the stabilizer of $\infty$
in $\rho^{-1}G\rho$
contains a maximal subgroup of $G$ consisting of elements of the form
$\gamma_x=\left(\begin{array}{cc}1 & x\\0 & 1\end{array}\right)$, where $x\in L$
for some fractional ideal $L$
of $\A$. Such an element $\gamma_x$ acts on $\Omega$ as a translation by $x$,
therefore the first condition
gives $f(z+x)=f({\gamma}_x z)=f(z)$. Compared with the classical situation,
the rigid analytic
function $t_L(z)$ of (\ref{e:3.1}) serves as a parameter at the infinity
point $\infty$. So $f$
is holomorphic at $p$ if and only if $f$ can be expanded as a series
in terms of the parameter
$t_L(z)$:
\begin{equation}\label{e:3.2}
f|{[\rho]_{k,m}}(z)=\sum\limits_{i\ge 0}c_i t_L(z)^i, \quad\text{ where $c_i\in {\C}_\infty$}.
\end{equation}
If the expansion coefficient $c_0=0$ in (\ref{e:3.2}) for all cusps of $G$, then $f$ is called
a cusp form of $G$ with weight $k$ and type $m$.

\begin{rem}\label{rem:3.1}
Although the above definition of modular forms is stated for the
general case when $\bk$ is the field of functions on a complete,
geometrically irreducible curve over the field $\Fq$ and $\A$ is the
ring of regular functions away from the point $\infty$, we only
consider the case when the curve is the projective line ${\mathbb
P}^1_{\Fq}$, so we assume ${\bk}={\Fq}(T)$ and
${\A}={\Fq}[T]$ throughout this paper.
\end{rem}

\begin{example}\label{example:3.2}
The Eisenstein series $E_k(z)$ is defined as
\begin{equation}\label{e:3.3}
E_k(z)=\sum\limits_{(0,0)\ne (c,d)\in {\A}^2}\frac{1}{(cz+d)^k}
\end{equation}
for an integer $k\ge 0$. If $k\not\equiv 0\mod (q-1)$, then $E_k(z)$ is identically equal to $0$.
For $k\equiv 0 \mod (q-1)$, the Eisentein series $E_k(z)$ is a modular form of the arithmetic
subgroup $\Gamma={\GL2}(\A)$ with weight $k$ and type $0$. But it is not a cusp form.
See \cite{Go1980b} for the detail.
\end{example}

\begin{example}\label{example:3.3}
This is an example given by Gekeler \cite{Ge1988}.
Let
\[
H=\left\{\left(\begin{array}{cc}a&b\\0&1\end{array}\right):\, \, a\in {\Fq}^*, b\in {\A}\right\}
\]
be a subgroup of $\Gamma$, and $t(z)$ be the rigid analytic function defined as (\ref{e:3.1}) for the
rank $1$ lattice ${\A}\bar{\pi}$, where the constant $\bar{\pi}\in {\C}^*_\infty$ is some ``period".
Then the Poincar\'e series
\begin{equation}\label{e:3.4}
P_{k,m}(z)=\sum\limits_{\gamma\in H\backslash\Gamma}(\det(\gamma))^{m}j(\gamma,z)^{-k}t^m(\gamma z)
\end{equation}
is a cusp form of $\Gamma$ with weight $k$ and type $m \mod (q-1)$,
where $j(\gamma,z)$ is defined as in
(\ref{e:3.1.5}).
\end{example}

\begin{example}\label{example:3.4}
Let $L\subset {\C}_\infty$ be an $\A$-lattice of ${\C}_\infty$ of rank $2$. Such a lattice $L$ is
associated with an ${\Fq}$-linear function
\[
e_L(z)=z\prod\limits_{0\ne \lambda\in L}\left(1-\frac{z}{\lambda}\right)
\]
which is rigid analytic on $\Omega$. Then there exists a ring homomorphism
\[
\begin{array}{llll}
\phi^L: &{\A}&\to &{\C_\infty}\{\tau\}=
\{\sum_{i\ge 0}a_i\, \tau^i:\,\, a_i\in {\C}_\infty\text{ for each integer $i\ge 0$}\}\\
        &a &\mapsto &\phi^L_a
\end{array}
\]
such that
\begin{equation}\label{e:3.5}
\phi^L_a\,\cdot\, e_L = e_L\,\cdot\, a
\end{equation}
for any $a\in {\A}$, where $\tau(x)=x^q$ for any $x\in {\C}_\infty$
is the Frobenius map. In the above equation (\ref{e:3.5}), the dot
notion ``\,$\cdot$\," denotes the composition map, $a$ is
regarded as the multiplication map from ${\C}_\infty$ to itself, and
$e_L: {\C}_\infty\to {\C}_\infty$ is the additive map given by the
function $e_L(z)$.

The ring homomorphism $\phi^L$ is a Drinfeld module
of rank $2$ over ${\C}_\infty$ associated to the $\A$-lattice $L$ of
${\C}_\infty$. For $0\ne a\in {\A}$, we can see that $\phi^L_a(z)$ is a
polynomial in $z$ and is given as
\[
\phi^L_a(z)= az\prod\limits_{0\ne \lambda\in a^{-1}L/L}
\left( 1-\frac{z}{e_L(\lambda)}\right)
\]
by checking the zeroes of $e_L(az)$ and $\phi^L_a(e_L(z))$ and the coefficients of
the first terms
of the expansions of these two rigid analytic functions in terms of $z$. The
homomorphism $\phi$ is determined by
\begin{equation}\label{e:3.6}
\phi^L_T=T\tau^0 + g_L\,\tau + \Delta_L\tau^2.
\end{equation}
We notice that $\lambda L$ is also an
$\A$-lattice of ${\C}_\infty$ for $\lambda\in {\C}^*_\infty$
and $e_{\lambda L}(\lambda z)=\lambda\, e_L(z)$, therefore
\[
\begin{array}{ll}
\phi^{\lambda L}_T(\lambda\, e_L(z))&=\phi^{\lambda L}_T(e_{\lambda L}(\lambda\, z))\\
&=e_{\lambda L}(T\lambda\, z)\\
&=\lambda\, e_L(Tz)\\
&=\lambda\, \phi^L_T(e_L(z)).
\end{array}
\]
Putting the above equation into equation (\ref{e:3.6}) for $\phi^L_T$ and $\phi^{\lambda L}_T$, we get
\[
T\lambda\, e_L(z) + g_{\lambda L}\lambda^qe_L(z)^q + \Delta_{\lambda L}\lambda^{q^2}e_L(z)^{q^2}
=\lambda\, Te_L(z) + \lambda\, g_{L} e_L(z)^q + \lambda\, \Delta_L e_L(z)^{q^2}.
\]
Thus we have
\begin{equation}\label{e:3.7}
g_{\lambda L} = \lambda^{1-q} g_L, \qquad \Delta_{\lambda L}=\lambda^{1-q^2} \Delta_L.
\end{equation}
An $\A$-lattice of ${\C}_\infty$ of rank $2$ can be written as
$L={\A}\omega_1\oplus {\A}\omega_2$ $\sim$ ${\A}z\oplus {\A}$ with
$\omega_1,\omega_2\in {\C}^*_\infty$ and $z=\omega_1/\omega_2\in
\Omega$ (because $L$ is discrete in the topology of ${\C}_\infty$).
We define for the lattice $L_z={\A}z\oplus {\A}\subset {\C}_\infty$
with $z\in \Omega$
\begin{equation}\label{e:3.8}
g(z)=g_{L_z}, \qquad \Delta(z)=\Delta_{L_z}.
\end{equation}
From the above equation (\ref{e:3.8}), it is not difficult to see that for
$\gamma=\left(\begin{array}{cc}a&b\\c&d\end{array}\right)\in \Gamma$
\[
g(\gamma z)=(cz+d)^{q-1} g(z), \qquad \Delta(\gamma z)=(cz+d)^{q^2-1}\Delta(z).
\]
The two functions $g(z)$ and $\Delta(z)$ on $\Omega$ can be
expressed in terms of Eisenstein series $E_k(z)$ (see \cite{Go1980a}
or Section 2, Chapter II of \cite{Ge1986}), so they are rigid
analytic functions on $\Omega$ and are holomorphic at the cusp
$\infty$ of $\Gamma$. Therefore $g(z)$ is a modular form of $\Gamma$
of weight $q-1$ and type $0$ and $\Delta(z)$ is a modular form of
$\Gamma$ of weight $q^2-1$ and type $0$.

The function $\Delta(z)$ on the Drinfeld's upper half plane $\Omega$ is called the Drinfeld discriminant. We refer \cite{Ge1986} and
Chapter 4 of \cite{Go1998} for more details and related topics
of Drinfeld modules.
\end{example}

\begin{theorem}\label{thm:3.1}
We have the following results with respect to the modular forms of $\Gamma$:
\begin{itemize}
\item[(1)]{\rm (Goss, \cite{Go1980a})} The space of modular forms with respect to
the group $\Gamma$ of type $0$ (and of all weights) is the polynomial ring ${\C}_\infty[g,\Delta]$.
\item[(2)]{\rm (Gekeler, \cite{Ge1988})} The space of modular forms with respect to
the group $\Gamma$ (of all weights and all types) is the polynomial ring ${\C}_\infty[g,P_{q+1,1}]$, where
$P_{q+1,1}$ is given in (\ref{e:3.4}). Moreover, $\Delta=-(P_{q+1,1})^{q-1}$.
\item[(3)]{\rm (Goss, \cite{Go1980a}; also see \cite{Ge1988})} The following two equalities hold:
\begin{align}
&g(z)=(T^q-T)E_{q-1}(z),\notag\\
&\Delta(z)=(T^{q^2}-T)E_{q^2-1}(z)+(T^q-T)^q E_{q-1}(z)^{q+1}. \label{e:3.9}
\end{align}
\end{itemize}
\end{theorem}

%==============================================================================================================

\section{Characteristic $p$ valued measures on ${\P1}({\bk}_\infty)$}\label{cpvm}

For $a\in {\bk}_\infty$ and $0\ne \rho\in {\bk}_\infty$, let $B_a(|\rho|)=\{x\in {\bk}_\infty:\,\,
|x-a|\le |\rho|\}$ be the closed ball of radius $|\rho|$ centered at $a$, and let $B_{\infty}(|\rho|)
=\{x\in {\bk}_{\infty}: \,\, |x|\ge |\rho^{-1}|\}\cup \{\infty\}$, which is viewed as a closed ball
of radius $|\rho|$ centered
at $\infty$. A function $f: B_a(|\rho|)\to {\C}_{\infty}$ on the ball $B_a(|\rho|)$ is said to be analytic
if $f$ can be expanded as a Taylor series
\begin{equation}\label{e:cpvm1}
f(x)=\sum\limits_{n=0}^{\infty} c_n \left(\frac{x-a}{\rho}\right)^n \quad \text{$(c_n\in {\C}_\infty$ for $n\ge 0$)}
\end{equation}
which is convergent for any $x\in B_a(|\rho|)$, i.e., $c_n\to 0 $ as $n\to \infty$. And a function
$f:B_{\infty}(|\rho|)\to {\C}_{\infty}$ is said to be analytic on $B_{\infty}(|\rho|)$
if $g(x):=f(\frac{1}{x})$ is analytic on the ball $B_0(|\rho|)$.

Let $S\subset {\P1}({\bk}_\infty)$ be an open compact nonempty subset. Hence $S$ is a finite disjoint union of closed
balls of positive radii.

A function $f: S\to {\C}_{\infty}$ is said to be locally analytic (of order $l$, respectively)
at the point $a\in S$ if $f$ is analytic on some closed ball $B_a$ which is centered at $a$ and of positive
radius (at least $|\pi|^l$, respectively). And $f$ is said to be locally analytic (of order $l$, respectively) on
$S$ if $f$ is locally analytic (of order $l$, respectively) at every point of $S$.
The space of all locally analytic functions (of order $l$, respectively) (taking values in ${\C}_{\infty}$) on $S$ is
denoted by $LA(S)$ ($LA_l(S)$, respectively).

\begin{rem}\label{rem:cpvm0}
The above definition of local analyticity of order $l$ is a little
ambiguous, since the balls $B_a(|\pi|^l)$ may not be contained in
$S$ for a given positive integer $l$. But we have assumed that $S$
is open compact, therefore $B_a(|\pi|^l)\subset S$ as long as the
integer $l$ is sufficiently large (say $|\pi|^l$ is less than or equal to the
smallest value among all radii in a decomposition of $S$ into finite disjoint
union of closed balls of positive radii). So the definition makes sense, and we
always assume that such an integer $l$ in the above is sufficiently large.
\end{rem}

The space $LA_l(S)$ is equipped with a norm $||\cdot||$ as follows. We have a decomposition of $S$ into a finite
disjoint union
\begin{equation}\label{e:cpvm2}
S=\bigsqcup_i B_i,
\end{equation}
where each $B_i$ is a closed ball of radius $|\pi|^l$. We notice that if one of these closed balls, say $B_{i_0}$,
is centered at $\infty$, then $B_{i_0}=B_{\infty}(|\pi|^l)=\{x\in {\bk}_\infty:\,\,
|x|\ge |\pi|^{-l}\}\cup\{\infty\}$. Suppose $f\in LA_l(S)$. As every element of a closed ball
is a center under a non-Archimedean absolute value, $f$ has an expansion as equation (\ref{e:cpvm1}) on
each closed ball $B_i$ in the decomposition (\ref{e:cpvm2}):
\begin{equation}\label{e:cpvm3}
f|_{B_i}(x)=\sum\limits_{n\ge 0} c_{in} \left(\frac{x-a_i}{\pi^l}\right)^n,
\end{equation}
for any $a_i\in B_i$, and the expansion should be replaced by the following if $B_{i_0}$
is the closed ball centered at $\infty$:
\[
g_{i_0}(x)=\sum\limits_{n\ge 0} c_{i_0,\,n} \left(\frac{x}{|\pi|^l}\right)^n,
\quad g_{i_0}(x):=f|_{B_{i_0}}\left(\frac{1}{x}\right).
\]
Then we define
\begin{equation}\label{e:cpvm4}
||f||_{B_i}=\max\limits_{n\ge0}\{|c_{in}|\}, \quad\text{ and } ||f||=\max\limits_{i}\{||f||_{B_i}\}.
\end{equation}
To see (\ref{e:cpvm4}) is well defined, we need to show
\begin{lemma}\label{lemma:cpvm1}
$||f||_{B_i}$ is independent of the choices of the center $a_i$.
\end{lemma}
\begin{proof}
Suppose $b_i\in B_i$. Then $b_i$ is also a center of $B_i$.  And
$f|_{B_i}(x)$ has expansion (\ref{e:cpvm3}) at $a_i$, and has the
following expansion at $b_i$ as well:
\[
f|_{B_i}(x)=\sum\limits_{n\ge 0}
c'_{in}\left(\frac{x-b_i}{\pi^l}\right)^n, \quad\text{ where }
c'_{in}=\sum\limits_{j\ge n}
c_{ij}\binom{j}{n}\left(\frac{b_i-a_i}{\pi^l}\right)^{j-n}.
\]
As $|b_i-a_i|\le |\pi^l|$, $\left|\binom{j}{n}\right|\le 1$, and
$|c_{ij}|\le ||f||_{B_i}$, we see that $||f||'_{B_i}:=\max_{n\ge
0}\{|c'_{in}|\}\le ||f||_{B_i}$. In the same way, we have
$||f||_{B_i}\le ||f||'_{B_i}$, hence get the conclusion.
\end{proof}
It is easy to see that $LA_l(S)$ is a Banach space with the norm
defined above. And we have the following sequence of closed
inclusion of Banach spaces
\begin{equation*}
\cdots\hookrightarrow LA_l(S)\hookrightarrow
LA_{l+1}(S)\hookrightarrow LA_{l+2}(S)\hookrightarrow \cdots
\end{equation*}
and $LA(S)=\lim\limits_{\longrightarrow}
LA_l(S)=\mathop{\cup}\limits_{l} LA_l(S)$. We equip $LA(S)$ with the
topology of direct limit. In our case, a subset $W\subset LA(S)$ is
closed if and only if $W\cap LA_l(S)$ is closed in $LA_l(S)$ for
each sufficiently large integer $l$. This is equivalent to saying
that for any topological space $Z$, a map $f: LA(S)\to Z$ is
continuous if and only if $f|_{LA_l(S)}: LA_l(S)\to Z$ is continuous
for each sufficiently large $l$.

We denote by $M^*={\rm Hom}_{\C_\infty}(M, {\C}_\infty)$ the space
of all ${\C}_\infty$-linear continuous maps from $M$ to
${\C}_\infty$ for a given topological vector space $M$ over
${\C}_\infty$.

\begin{rem}\label{rem:cpvm0.5}
Suppose $(M, ||\cdot||)$ is a separable ${\C}_\infty$-Banach space.
For any map $f: M\to {\C}_\infty$ we define $||f||=\sup\limits_{0\ne
x\in M}\left\{|f(x)|/||x||\right\}$. Then
\begin{equation*}
\begin{split}
M^*&=Hom_{{\C}_\infty}(M,{\C}_\infty) \\
   &=\{f: f \text{ is ${\C}_\infty$-linear and continuous from $M$ to ${\C}_\infty$ }\} \\
   &=\{f: f \text{ is ${\C}_\infty$-linear and } ||f||< \infty \}.
\end{split}
\end{equation*}
\end{rem}

The elements of $LA(S)^*$ are called tempered distributions on $S$
(with values taken in ${\C}_\infty$), and we write $(f,\mu):=\mu(f)$
for $\mu\in LA(S)^*$ and $f\in LA(S)$.

\begin{rem}\label{rem:cpvm1}
A ${\C}_\infty$-valued distribution $\mu$
on $S$ is a ${\C}_\infty$-linear function from the set
\[
P(S)=\{U:U\subset S\text{ is open and compact }\}
\]
to ${\C}_\infty$ which satisfies the finite additivity property: if
$U,V\in P(S)$ and $U\cap V=\emptyset$, then
$\mu(U\cup V)=\mu(U)+\mu(V)$. If $\{\mu(U):U\in P(S)\}$ is bounded,
then $\mu$ is called a measure on $S$. Since locally constant functions on $S$
are locally analytic, we see that an element $\mu \in LA(S)^*$ can be assigned
for any open-compact subset $U\subset S$
\[
\mu(U):= (\xi_U(x), \mu),
\]
where $\xi_U(x)$ is the characteristic function of the subset $U$. Then it is easy to
see that $\mu$ satisfies the finite additivity property, thus $\mu$ is a distribution on $S$.
\end{rem}

\begin{defn}\label{defn:cpvm1}\begin{itemize}
\item[(1)]\  For integers $l,j\geq 0$, and $a\in S$, we define the following
functions on $S$:
\begin{equation*}
\k(a,l;j;x)=
\begin{cases}
({x-a})^j, &\qquad\text{ if $x\in B_a(|\pi|^l)$, } \\
0,       &\qquad\text{ if $x\notin B_a(|\pi|^l)$, }
\end{cases}
\end{equation*}
for $B_a(|\pi|^l)\subset S$, provided that $a\ne \infty$, and
\begin{equation*}
\k(\infty,l;j;x)=
\begin{cases}
({\frac{1}{x}})^j, &\qquad\text{ if $x\in B_{\infty}(|\pi|^l)$, } \\
0,       &\qquad\text{ if $x\notin B_\infty(|\pi|^l)$, }
\end{cases}
\end{equation*}
for $B_{\infty}(|\pi|^l)\subset S$, where
$B_{\infty}(|\pi|^l)=\{x\in {\bk}_\infty:\,\, |x|\ge
|\pi^{-1}|^l\}$. And in what follows, we understand the term
$(x-a)^j$ on $B_a(|\pi|^l)$ needs to be replaced by $(1/x)^j$ if
$a=\infty$.
\item[(2)]\ The ${\C}_\infty$-vector spaces $P_l^{(n)}, n\geq 0$, are defined as:
\[
P_l^{(n)}=\sum\limits_{\substack{a\in S \\ 0\leq j\leq n}}
  {\C}_\infty\,\k(a,l;j;x),
\]
the vector space generated by all the functions
$\k(a,l;j;x)$, with $0\leq j\leq n, a\in S$, over ${\C}_\infty$. And we put
\[
P^{(n)}=\lim\limits_{\longrightarrow} P_l^{(n)},
\]
where the direct limit is taken with respect to the injective maps $P_l^{(n)}\to P_{l+1}^{(n)}$.
\par
\item[(3)]\ Denote
\[
P^{(\infty)}=\lim\limits_{\longrightarrow} P^{(n)}, \quad\text{ and } \quad
P_l^{(\infty)}=\lim\limits_{\longrightarrow} P_l^{(n)},
\]
where the direct limits are taken over the injective maps $P^{(n)}\to P^{(n+1)}$ and
$P_l^{(n)}\to P_l^{(n+1)}$ respectively.
$P^{(\infty)}$ is in fact the space of all locally polynomial functions over
$S$, and $P_l^{(\infty)}$ those defined for balls of radius $|\pi|^l$ (which are contained in $S$).
\end{itemize}
\end{defn}

As remark \ref{rem:cpvm1} indicates, an element $\mu\in LA(S)^*$ is a distribution on $S$, thus it
is natural to write
\[
\int_S f(x)d\mu(x):= (f,\mu)
\]
for $f\in LA(S)$. And we also write for an open compact subset $U\subset S$
\[
\int_U f(x) d\mu(x) = \int_S f(x)\xi_U(x) d\mu(x)
\]
provided that $f(x)$ is locally analytic over $U$, and we extend $f$
by $0$ outside of $U$ in
the second integral. We notice that for $\mu\in LA(S)^*$,
the set $\{\mu(U)=(\xi_U(x),\mu):\,\,
U\subset S \text{ open compact}\}$ does not determine $\mu$. Instead we have

\begin{prop}\label{prop:cpvm1} \begin{itemize}
\item[(1)]\ $\mu\in (P_l^{(\infty)})^*=Hom_{{\C}_\infty}(P_l^{(\infty)},{\C}_\infty)$ can be extended
to an element of $(LA_l(S))^*$ if and only if
\[
\left|\int_{B_a(|\pi|^l)} \k(a,l;j;x)d\mu(x)\right|
=\left|\int_S \k(a,l;j;x)d\mu(x)\right|
\leq C\cdot |\pi|^{lj},
\]
for any $j\geq 0$ and any $B_a(|\pi|^l)\subset S$, where $C$ is a constant depending only on $\mu$.
\par
\item[(2)]\ $\mu\in (P^{(\infty)})^*=Hom_{{\C}_\infty}(P^{(\infty)},{\C}_\infty)$ can be extended
to an element of $(LA(S))^*$ if and only if
\[
\left|\int_{B_a(|\pi|^l)} \k(a,l;j;x)d\mu(x)\right|
=\left|\int_S \k(a,l;j;x)d\mu(x)\right|
\leq C(l)\cdot |\pi|^{lj},
\]
for any sufficiently large integer $l$, any $j\geq 0$, and any $B_a(|\pi|^l)\subset S$,
where $C(l)$ is a constant depending only on $l$ and $\mu$.
\end{itemize}
\end{prop}
\begin{proof}
Since the topology on $LA(S)$ is induced from those of all
$LA_l(S)$'s, we need only prove (1). Suppose $\mu\in
(P_l^{(\infty)})^*$ can be extended to an element of $LA_l(S)$,
which is still denoted by $\mu$. By remark \ref{rem:cpvm0.5},
\[
\begin{split}
&\left|\int_{B_a(|\pi|^l)}
\k(a,l;j;x)d\mu(x)\right|=\left|\int_{B_a(|\pi|^l)}(x-a)^j d\mu(x)\right|\\
&\le ||\mu||\cdot |\pi|^{lj}\cdot
\left|\left|\left(\frac{x-a}{\pi^l}\right)^j\right|\right|\\
&\le ||\mu||\cdot |\pi|^{lj},
\end{split}
\]
if ${\infty}\not\in B_a(|\pi|^l)$. And similarly if ${\infty}\in
B_a(|\pi|^l)$ (therefore $a=\infty$ by the assumption on the notation
$B_a(|\pi|^l)$).

Conversely, suppose the condition holds, and $f\in LA_l(S)$. Then we
decompose $S=\bigsqcup_i B_i$ as a finite disjoint union of closed
balls of radius $|\pi|^l$. On each $B_i$, the function $f$ can be
expanded as
\[
f|_{B_i}(x)=\sum\limits_{n\ge
0}c_{in}\left(\frac{x-a_i}{\pi^l}\right)^n, \quad \text{ with
$a_i\in B_i$, and $c_{in}\to 0$ as $n\to \infty$}
\]
(if $\infty\in B_i$, the function $f(x)=g(1/x)$ for some analytic
function $g$ on $B_0(|\pi|^l)$, and we get the expansion of $f$ in terms
of the parameter $1/x$ by the expansion of $g$, then we proceed similarly).
Therefore we extend $\mu\in (P_l^{(\infty)})^*$ by defining
\[
(f,\mu)=\sum\limits_i\sum\limits_{n\ge 0} {\pi^{-ln}}c_{in}
\int_{B_i} (x-a_i)^nd\mu(x),
\]
which is convergent under the given assumption. The assumption on
$\mu\in (P_l^{(\infty)})^*$ also implies that the extension of $\mu$ to
$LA_l(S)$ is continuous.
\end{proof}

\begin{rem}
In Proposition \ref{prop:cpvm1}, the estimate on the condition for
$\mu\in (P^{(\infty)})^*$ to
be extendable to an element of $(LA(S))^*$
is based on the topology of the space of locally analytic functions.
In the following definition, we are going to specify the constant
$C(l)$ in (2) of Proposition \ref{prop:cpvm1} and define the
``$h$-admissible" measures, which will be proved dual to the $C^h$-functions.
\end{rem}

\begin{defn}\label{defn:cpvm2}
For a non-negative integer $h$, a linear functional $\mu\in (P^{(h)})^*$ is called an $h$-admissible
measure on $S$, if it satisfies the following growth condition:
\begin{equation}\label{e:cpvm5}
\left|\int_{B_a(|\pi|^l) } (x-a)^jd\mu(x)\right| \leq C\cdot |\pi|
^{l(j-h)}
\end{equation}
for $0\leq j\leq h$, any $l$ sufficiently large, and any $B_a(|\pi|^l)\subset S$,
where $C$ is a constant depending only on $\mu$. If $a=\infty$, then
(\ref{e:cpvm5}) should be understood as
\begin{equation*}
\left|\int_{B_a(|\pi|^l) } \left(\frac{1}{x}\right)^jd\mu(x)\right| \leq C\cdot |\pi|
^{l(j-h)}.
\end{equation*}
\end{defn}

\begin{rem}\label{rem:cpvm4}
This definition is a little bit different from~\cite{Vi1976}
and~\cite{Vi1985} by Vishik. The right side
of (\ref{e:cpvm5}) is $o(1)\cdot |\pi|^{l(j-h)}$ in Vishik's papers.
This change is made in order to relate the $h$-admissible measures with $C^h$-functions on $S$
instead of functions satisfying the ``$h$-th order Lipschitz'' condition.
\end{rem}
It is clear from the above definition that the $0$-admissible measures
on $S$ are exactly
the measures on $S$. In general, the $h$-admissible measures on $S$
are dual to the space of
$C^h$ functions on $S$. Since $\bk_\infty$ is of characteristic $p>0$,
the definition of differentiability
of a function is a little different to the usual one in the case of
characteristic $0$. At first, we
consider the functions on closed balls of ${\P1}({\bk}_\infty)$
and follow Schikhof's definition of $C^n$-functions \cite{Sc1984}.

\begin{defn}\label{defn:cpvm3}
Let $B\subset {\bk}_\infty$ be a closed ball of positive radius. For an integer $n>0$, set
\[
\GD^n B=\{(x_1,x_2,\cdots,x_n)\in B^n: x_i\neq x_j \text{ if $i\neq j$ } \}.
\]
The $n$-th order difference quotient $\Phi_n f: \GD^{n+1}{B} \to {\C}_\infty$
of a function $f: {B} \to {\C}_\infty$ is inductively defined by
\begin{equation*}
\begin{split}
&\ \Phi_0 f = f, \\
&\ (\Phi_n f)(x_1,x_2,\cdots,x_{n+1})=
              \frac{(\Phi_{n-1} f)(x_1,x_3,\cdots,x_{n+1})-
                   (\Phi_{n-1} f)(x_2,x_3,\cdots,x_{n+1})} {x_1-x_2}
\end{split}
\end{equation*}
for $n\geq 1$ and $ (x_1,x_2,\cdots,x_{n+1}) \in \GD^{n+1}{B}$. The function $f$ is
called a $C^n$-function on ${B}$ if $\Phi_n f$ can be extended to a continuous
function $\overline{\Phi_n f}: B^{n+1} \to {\C}_\infty $.
For $B=B_{\infty}(|\pi|^{l})$, we
change the parameter of $B$ by the transform $\phi: x\to 1/x$ and make a similar definition.
Since $\phi$ is one-to-one and continuous on $B$, the definition of
differentiability of functions
on $B'\subset B_{\infty}(|\pi|^{l})$ for $B'$ a closed ball not containing
$\infty$ agrees with
that on $B_{\infty}(|\pi|^{l})$. For convenience,
it is understood throughout this paper
that such a transform is automatically applied to change the
parameter whenever the closed ball $B$ is centered at $\infty$.
\end{defn}

The space of
$C^n$-functions from $B$ to ${\C}_\infty$ is denoted by $C^n(B)$.
For $f\in C^n(B)$ and $x\in B$, set
\[
D_n f(x) = \overline{\Phi_n f}(x,x,\cdots,x) \text{ for $n>0$ }, \quad\text{ and }\quad
D_0 f(x) = f(x).
\]
$D_n f$ is the $n$-th order hyper-derivative of the function $f$. We
have the following characterization of $C^n$-functions on
$B$, in terms of Taylor expansions:
\begin{theorem}[Section 29 \& Section 83, \cite{Sc1984}]\label{thm:cpvm1}
Let $B$ be a closed ball of positive radius. If $f \in C^n(B)$, then for all $x,y \in B$,
\begin{equation*}
\begin{split}
f(x)&=f(y)+\sum_{j=1}^{n-1}(x-y)^jD_jf(y)+(x-y)^n\overline{(\Phi_nf)}
       (x,y,\cdots,y)\\
    &=f(y)+\sum_{j=1}^{n}(x-y)^jD_jf(y)+(x-y)^n\bigl (\overline{(\Phi_nf)}
       (x,y,\cdots,y)-D_nf(y) \bigr ).
\end{split}
\end{equation*}
Conversely, let $f$ be a function on $B$. If there exist
continuous functions $\l_1, \cdots, \l_{n-1}: B \to {\C}_\infty$ and
$\Lambda_n: B\times B \to {\C}_\infty$ such that
\[
f(x)=f(y)+\sum_{j=1}^{n-1}(x-y)^j\l_j(y)+(x-y)^n\Lambda_n(x,y)
  \qquad (x,y \in {\A_\infty}),
\]
then $f\in C^n(B)$.
\end{theorem}

$C^n(B)$ is a Banach space over ${\C}_\infty$ with the norm $||\cdot||_n$ defined by
\begin{equation}\label{e:cpvm6}
||f||_n=\max\left\{||\overline{(\Phi_0f)}||_{\infty},
||\overline{(\Phi_1f)}||_{\infty}, \cdots,
||\overline{(\Phi_nf)}||_{\infty}\right\} \quad\text{ for $f\in C^n({\A}_\infty)$ }
\end{equation}
where the functions $\overline{(\Phi_jf)}$ for $0\le j\le n$ are as in Definition \ref{defn:cpvm3},
and $||\cdot||_{\infty}$ is the sup norm of the functions on the respective topological spaces.
As the spaces $B^j=B\times B\times \cdots \times B$ ($j$ times,
$1\le j\le n+1$) are compact, the right side of (\ref{e:cpvm6}) is
finite. On the notation of the norm, we use $||\cdot||_{C^n(B)}$ for $||\cdot||_n$ if any confusion
may occur.

In the case $B={\A}_\infty$. The functions in $C^n({\A}_\infty)$ (as well as $LA_n({\A_\infty})$ and $LA({\A_\infty})$)
can be described by the Carlitz basis $\{G_n(x)\}_{n\ge 0}$, where $G_n(x)$ is a polynomial of degree $n$
for each integer $n\ge 0$, they are defined in the following. Set
\begin{itemize}
\item $[i]={\pi}^{q^i}-\pi$ for $i$ positive integer;
\item $L_i=1$ if $i=0$; and $L_i=[i]\cdot [i-1]\cdots[1]$ \ if $i$ is
      a positive integer;
\item $D_i=1$ if $i=0$; and $D_i=[i]\cdot [i-1]^q\cdots [1]^{q^{i-1}}$
      \ if $i$ is a positive integer;
\item $e_i(x) = x$, if $i=0$; $e_i(x) = \prod_{{\a}\in{\Fq}[\pi],\, \deg_{\pi}(\a)<i}(x-\a)$,
    if $i$ is a positive integer;
\item $E_i(x)= e_i(x)/D_i$, for each non-negative integer $i$.
\end{itemize}
Thus we see that $e_i(x)$ and $E_i(x)$ are ${\Fq}$-linear polynomials of degree $q^i$.
For any non-negative integer $n$, write $n$ in $q$-digit
expansion: $n=n_0+n_1\,q+\cdots+n_s\, q^s$ with $0\leq n_i<q$, the Carlitz polynomial $G_n(x)$ is defined by
\[
G_n(x)=\prod_{i=0}^s\left(E_i(x)\right)^{n_i}.
\]

%\begin{prop}[Carlitz, also see]\label{prop:cpvm2}
%For each non-negative integer $n$,
%\begin{itemize}
%\item[(1)] the polynomial $G_n(x)$ gives a map
%\item[(2)] the following formula holds for $G_n(x)$:
%\begin{equation}\label{e:cpvm7}
%G_n(x+y)=\sum\limits_{i=0}^n \binom{n}{i} G_i(x) G_{n-i}(y).
%\end{equation}
%\end{itemize}
%\end{prop}

\begin{theorem}[Wagner~\cite{Wa1971}, also see \cite{Go1989} or
\cite{Ya2004}]\label{thm:cpvm2}
The Carlitz polynomials $G_n(x)$ ($n\ge 0$) constitute an
orthonormal basis of the Banach space $C({\A}_\infty)$ of continuous
functions from ${\A}_\infty$ to ${\C}_\infty$ with the sup norm,
that is, any $f\in C({\A}_\infty)$ can be expressed as
\[
f(x)=\sum\limits_{n=0}^\infty a_n G_n(x), \quad\text{ with
$a_n\in{\C_\infty}$, and $a_n\to 0$ as $n\to \infty$}.
\]
And the norm $||f||$ (which is also denoted by $||f||_0$ in
(\ref{e:cpvm6})) is given by
\[
||f||:=\sup\limits_{x\in {\A_\infty}}\{|f(x)|\}=\max\limits_{n\ge
0}\{|a_n|\}.
\]
\end{theorem}

\begin{theorem}[\cite{Ya1998}, \cite{Ya2004}]\label{thm:cpvm3}
For integers $j, l\ge 0$, let $\mu_{j,\,l}=\sum_{i=l+1}^\infty
[j/q^i]$.
\begin{itemize}
\item[(1)] The polynomials $\pi^{\mu_{j,\,l}}\, G_j(x)$ with $j\ge 0$ constitute an orthonormal basis of the Banach
space $LA_l({\A}_\infty)$, that is, any $f\in LA_l({\A}_\infty)$ can be expanded as
$f(x)=\sum_{j=0}^\infty a_j \pi^{\mu_{j,\,l}}\, G_j(x)$ with $a_j\to 0$ as $j\to \infty$,
and $||f||_{LA_l({\A}_\infty)}=\max_{j} \{|a_j|\}$.
\item[(2)] Let $f(x)=\sum_{j=0}^\infty a_j\, G_j(x)$ be a continuous function from ${\A}_\infty$
to ${\C}_\infty$. Then $f\in C^n({\A}_\infty)$ if and only if $\lim_{j\to\infty} |a_j|j^{\,n}=0$.
\end{itemize}
\end{theorem}

From (2) of the above theorem, we can define a norm $||\cdot||'_n$ on $C^n({\A}_\infty)$ by
\begin{equation*}
||f||'_n:= \max\limits_{j} \{|a_j|\cdot |\pi^{-n[\log_q j\,]}\} \quad\text{
for } f(x)=\sum\limits_{j\ge 0} a_j\, G_j(x) \in C^n({\A}_\infty)
\end{equation*}
so that $(C^n({\A}_\infty), ||\cdot||'_n)$ is a Banach space over ${\C}_\infty$.

\begin{rem}\label{rem:cpvm5}
The norm $||\cdot||'_n$ and the norm $||\cdot||_n$ defined in
(\ref{e:cpvm6}) (take $B={\A}_\infty$) on $C^n({\A}_\infty)$ are
equivalent, that is, there exist constants $\lambda_1>0$ and
$\lambda_2>0$ such that $\lambda_1||\cdot||_n \le ||\cdot||'_n \le
\lambda_2 ||\cdot||_n$. This can be seen through the computation of
the difference quotients of the Carlitz polynomials $G_n(x)$, which
is carried out in the proof of Theorem 5.1 in \cite{Ya2004}.
\end{rem}

Due to the above remark, we will use the same notation $||\cdot||_n$ for the two norms $||\cdot||_n$
and $||\cdot||'_n$ on $C^n({\A}_\infty)$.

\begin{cor}\label{cor:cpvm1}
For $n\ge 0$ an integer,  $(C^n({\A}_\infty), ||\cdot||_n)$ is a ${\C}_\infty$-Banach
space with an orthonormal basis $\{\pi^{n[\log_q j\,]}G_j(x)\}_{j\ge 0}$.
\end{cor}

On a closed ball $B=B_a(|\pi|^{l_0})$, a function $f$ can be expressed as
$f(x)=g((x-a)/{\pi^{l_0}})$ for some function $g$ on ${\A}_\infty$, and $f\in C^n(B)$ if and only if
$g\in C^n({\A}_\infty)$. Hence we get the following

\begin{cor}\label{cor:cpvm2}
Let $B=B_a(|\pi|^{l_0})$. For $n\ge 0$ an integer, $(C^n(B), ||\cdot||_n)$ is a ${\C}_\infty$-Banach
space with an orthonormal basis $\{\pi^{n[\log_q j\,]}\,G_j((x-a)/{\pi^{l_0}})\}_{j\ge 0}$.
\end{cor}

\begin{cor}\label{cor:cpvm3}
Let $B=B_a(|\pi|^{l_0})$. The polynomials $\pi^{\mu_{j,\,l}} G_j((x-a)/{\pi^{l_0}})$ with $j\ge 0$ constitute
an orthonormal basis of the Banach space $LA_l(B)$, where $\mu_{j,\,l}$ is as in Theorem \ref{thm:cpvm3}.
\end{cor}

To study the $C^n$-functions on the open compact subset $S\subset {\P1}({\bk}_\infty)$, we
fix a decomposition (\ref{e:cpvm2}): $S=\bigsqcup_i B_i$, where each $B_i$ is a closed ball of
radius $|\pi|^{l_0}$. A function $f$ on $S$ is called of $C^n$ if and only if $f|_{B_i}$ is
of $C^n$, and the space of all $C^n$-functions on $S$ is denoted by $C^n(S)$.
We define a norm $||\cdot||_n$ on $C^n(S)$ by
\[
||f||_n =\max\limits_i\{||f||_{C^n(B_i)}\}.
\]
Then $(C^n(S), ||\cdot||_n)$ is a Banach space over ${\C}_\infty$.

\begin{lemma}\label{lemma:cpvm2}
Let $\mu_{n,l}=\sum\limits_{i\geq l+1}\left [
n/{q^i}\right ]$. Then
\[
\mu_{n,\,l}\geq \frac n{(q-1)q^l}+l-\log_qn-1.
\]
\end{lemma}
\begin{proof}
Express $n$ in $q$-digit expansion: $n=n_wq^w+\cdots+n_1q+n_0$ with
$n_w\neq 0$. Let $s(n,l)=n_w+\cdots+n_{l+2}+n_{l+1}$,
and $t(n,l)=n_lq^l+\cdots+n_1q+n_0$. We have
\begin{equation*}
\begin{split}
&\left [n/{q^{l+1}}\right ] =n_wq^{w-l-1}+\cdots+n_{l+2}q+n_{l+1}, \\
&\left [n/{q^{l+2}}\right ] =n_wq^{w-l-2}+\cdots+n_{l+2}, \\
&\cdots \cdots \cdots \cdots \\
&\left [n/{q^{l+w}}\right ] =n_w.
\end{split}
\end{equation*}
Add up the above equations, column by column on the right sides, then
\begin{equation*}
\begin{split}
\mu_{n,\,l}&=n_w\frac{q^{w-l}-1}{q-1}+n_{w-1}\frac{q^{w-l-1}-1}{q-1}
            +\cdots+n_{l+1}\frac{q-1}{q-1} \\
         &=\frac{q^{-l}(n-t(n,\,l))-s(n,\,l)}{q-1} \\
         &=\frac{n-t(n,\,l)-q^ls(n,\,l)}{(q-1)q^l}.
\end{split}
\end{equation*}
It's easily seen that $s(n,\,l)\leq (q-1)(w-l)\leq (q-1)(\log_qn-l)$
and $t(n,\,l)\leq (q-1)q^l$, therefore
$\mu_{n,\,l}\geq \frac n{(q-1)q^l}+l-\log_qn-1$.
\end{proof}

\begin{prop}\label{prop:cpvm3}
Let $h$ be a non-negative integer. If $\mu\in(C^h(S))^*$, then $\mu$ is an $h$-admissible measure on $S$.
\end{prop}
\begin{proof}
We need to prove the inequality (\ref{e:cpvm5}). Write
$B_i=B_{a_i}(|\pi|^{l_0})$ for each $i$ in the decomposition $S=\bigsqcup_i B_i$.
Then any closed ball $B_a(|\pi|^l)\subset S$ is contained in $B_i$ for some $i$.
Hence $\left((x-a)/(\pi^{l}) \right)^j\cdot
\xi_{B_{a}(|\pi|^{l})}(x)$ is locally analytic of order $l$
and has norm 1 in the space $LA_{l}(B_i)$ or $LA_{l}(S)$. Corollary \ref{cor:cpvm3} implies
\begin{equation}\label{e:cpvm8}
\left(\dfrac{x-a}{\pi^{l}}\right)^j\xi_{B_{a}(|\pi|^{l})}(x)
=\sum_{n=0}^\infty a_n\,\pi^{\mu_{n,\,l}}G_n((x-a_i)/{\pi^{l_0}}),
\end{equation}
where $a_n\to 0$ as $n\to \infty$, $|a_n|\leq 1$ for all $n\geq 0$.
And $\mu_{n,\,l}\geq \dfrac{n}{(q-1)q^l}+l-1-\log_qn$, by Lemma \ref{lemma:cpvm2}.
Also Corollary \ref{cor:cpvm2} implies
\begin{equation*}
\left|\int_{B_i} \pi^{h[\log_q n]}G_n((x-a_i)/{\pi^{l_0}})d\mu(x)\right|
\le ||\mu||_{(C^h(B_i))^*}\le ||\mu||_{(C^h(S))^*}
\end{equation*}
since $\mu\in (C^h(S))^*$ and $C^h(B_i)\subset C^h(S)$. Therefore
\begin{equation}\label{e:cpvm9}
\begin{split}
v\left(a_n\pi^{\mu_{n,\,l}}\int_{B_i} G_n((x-a_i)/{\pi^{l_0}})d\mu(x)\right)
&\geq\frac{n}{(q-1)q^l}-1-\log_qn+l-C_1-(\log_qn)h
\end{split}
\end{equation}
where $C_1$ does not depend on $l$. A little calculation on the minimum value
of the function $m(x)=\frac{x}{(q-1)q^l}-(h+1)(\log_q x)$ gives
$\frac{n}{(q-1)q^l}-(h+1)(\log_q n)\ge C_2-l(h+1)$, where the constant $C_2$
does not depend on $l$ or $n$. Therefore the right side of (\ref{e:cpvm9}) is
greater than or equal to $C_3-lh$ with $C_3$ a constant not depending on $l$ or
$n$. Integrating the equation
(\ref{e:cpvm8}), we get the required estimation.
\end{proof}

Now we are going to show that the converse of Proposition \ref{prop:cpvm3} is
also true. Suppose $\mu$ is an $h$-admissible measure on $S$.

For each fixed integer $l$ which is sufficiently large, we have the finite disjoint decomposition
\[
S=\bigsqcup_{0\le i \le r(l)}B_i(|\pi|^l)
\]
where $r(l)$ is a positive integer depending on $l$, and fix a point $a_{l,i}$
inside each closed ball $B_i(|\pi|^l)$.

Let $f\in C^h(S)$. For any $a\in S$, there exists a $B_i(|\pi|^l)$ containing this element $a$.
From Theorem \ref{thm:cpvm1}, the function $f$ has a Taylor expansion on $B_i(|\pi|^l)$:
\begin{equation}\label{e:cpvm9.5}
f(x)=f(a)+Df(a)(x-a)+\cdots+D_hf(a)(x-a)^h+\Lambda_h(f;x,a)(x-a)^h
\end{equation}
where $\Lambda_h(f;x,a)$ is continuous on $B_i(|\pi|^l) \times B_i(|\pi|^l)$ and
$\lim_{x\to a}\Lambda_h(f;x,a) =0$. The Riemann sum of $f$ with
respect to the $h$-admissible measure $\mu$ is defined by
\begin{equation}\label{e:cpvm10}
R_l(f,\mu,\{a_{l,i}\})=\sum_{0\leq i\leq r(l)}
\sum_{j=0}^h D_jf(a_{l,i})(\k(l,j,a_{l,i};x),\mu)
\end{equation}

\begin{lemma}[Schikhof, Theorem 78.2 of \cite{Sc1984}]\label{lemma:cpvm3}
Let $h$ be any non-negative integer and $f\in C^h(S)$. Then $D_j
f\in C^{h-j}(S)$ and $D_iD_jf=\binom{i+j}{i}D_{i+j}f$ for $0\le i,
0\le j$, and $i+j\le h$.
\end{lemma}

\begin{lemma}\label{lemma:cpvm4}
The limit $\lim_{l\to \infty}R_l(f,\mu,\{a_{l,i}\})$ exists,
and does not depend on the choices
of $a_{l,i}\in B_i(|\pi|^l)$ for
any $l$ sufficiently large and $0\leq i\leq r(l)$. And we define
\begin{equation}\label{e:cpvm11}
\int_S f(x)d\mu(x)=\lim_{l\to \infty}R_l(f,\mu,\{a_{l,i}\}).
\end{equation}
\end{lemma}
\begin{proof}
The proof here is similar to Vishik's paper \cite{Vi1976}, where the integral is with
respect to the $h$-th order Lipshitz functions instead of $C^h$-function.
At first we prove the limit does not depend on the choices of
$\{a_{l,i}\}_{0\leq i\leq r(l)}$ if it exists. Let
$b_{l,i}\in B_i(|\pi|^l)$ be another choice for each $l$ and $i$.
Then
\begin{equation*}
\begin{split}
&R_l(f,\mu,\{b_{l,i}\}) \\
=&\sum_{0\leq i\leq r(l)}\sum_{0\leq j\leq h} D_jf(b_{l,i})
    (\k(b_{l,i},l;j;x),\mu(x)) \\
=&\sum_{0\leq i\leq r(l)}\sum_{0\leq j\leq h}\sum_{0\leq k\leq j}
    \binom jk D_jf(b_{l,i})(a_{l,i}-b_{l,i})^{j-k}
    (\k(a_{l,i},l;k;x),\mu(x)) \\
=&\sum_{0\le i\le r(l)}\sum_{k=0}^h\sum_{j=0}^{h-k}\binom{j+k}{k}
    D_{j+k}f(b_{l,i})(a_{l,i}-b_{l,i})^j
    (\k(a_{l,i},l;k;x),\mu(x)) .
\end{split}
\end{equation*}
But
\[
\sum_{j=0}^{h-k}\binom{j+k}kD_{j+k}f(b_{l,i})(a_{l,i}-b_{l,i})^j
=D_kf(a_{l,i})-\Lambda_h(D_kf;a_{l,i},b_{l,i})
(a_{l,i}-b_{l,i})^{h-k}
\]
by making an appropriate substitution in (\ref{e:cpvm9.5}) and
applying Lemma \ref{lemma:cpvm3}, therefore
\begin{equation*}
\begin{split}
&R_l(f,\mu,\{b_{l,i}\}) \\
=&R_l(f,\mu,\{a_{l,i}\})-\sum_{i=0}^{r(l)}\sum_{k=0}^h
    \Lambda_h(D_kf;a_{l,i},b_{l,i})(a_{l,i}-b_{l,i})^{h-k}
    (\k(a_{l,i},l;k;x),\mu(x)).
\end{split}
\end{equation*}
This proves $\lim\limits_{l\to \infty}
(R_l(f,\mu,\{b_{l,i}\})-R_l(f,\mu,\{a_{l,i}\}))=0$, by the condition
(\ref{e:cpvm5}) on $\mu$ and $\lim\limits_{l\to \infty}
\Lambda_h(D_kf;a_{l,i},b_{l,i})=0$ for each $k$ with $0\leq k\leq
h$. Therefore the limit $\lim\limits_{l\to \infty}
R_l(f,\mu,\{a_{l,i}\})$, if it exists, does not depend on the
choices of $a_{l,i}\in B_i(|\pi|^l)$ for each $i$ and $l$.

Now let $m>l$. Then $B_i(|\pi|^l)$ can be decomposed disjointly into
$q^{m-l}$ smaller balls for each $i$:
\[
B_i(|\pi|^l)=\underset{s(i)\in I(l,m,i)}{\bigsqcup}
    B_{s(i)}(|\pi|^m).
\]
Therefore
\begin{equation*}
\begin{split}
&R_l(f,\mu,\{a_{l,i}\}) \\
=&\sum_{i=0}^{r(l)}\sum_{j=0}^h D_jf(a_{l,i})
  \int_{B_i(|\pi|^l)} (x-a_{l,i})^j d\mu(x) \\
=&\sum_{i=0}^{r(l)}\sum_{j=0}^h D_jf(a_{l,i})\sum_{s(i)\in I(l,m,i)}
  \sum_{k=0}^j\int_{B_{s(i)}(|\pi|^m)} \binom jk (a_{m,s(i)}-a_{l,i})^{j-k}
  (x-a_{m,s(i)})^kd\mu(x) \\
=&\sum_{i=0}^{r(l)}\sum_{s(i)\in I(l,m,i)} \sum_{k=0}^h \sum_{j=k}^h
  \binom jk D_jf(a_{l,i})(a_{m,s(i)}-a_{l,i})^{j-k}
  \int_{B_{s(i)}(|\pi|^m)} (x-a_{m,s(i)})^kd\mu(x) \\
=&\sum_{i=0}^{r(l)}\sum_{s(i)\in I(l,m,i)} \sum_{k=0}^h
  (D_kf(a_{m,s(i)})-\Lambda_h(D_kf;a_{m,s(i)},a_{l,i}))
  \int_{B_{s(i)}(|\pi|^m)} (x-a_{m,s(i)})^kd\mu(x).
\end{split}
\end{equation*}
Applying the definition of Riemann sum for $h$-admissible measures,
we get
\begin{equation}\label{e:cpvm13}
\begin{split}
&R_l(f,\mu,\{a_{l,i}\}) \\
=&R_m(f,\mu,\{a_{m,i}\})-
  \sum_{i=0}^{r(l)}\sum_{s(i)\in I(l,m,i)} \sum_{k=0}^h
  \Lambda_h(D_kf;a_{m,s(i)},a_{l,i})(a_{m,s(i)}-a_{l,i})^{h-k} \\
&\qquad \hspace*{100pt} \times \int_{B_{s(i)}(|\pi|^m)}
 (x-a_{m,s(i)})^kd\mu(x).
\end{split}
\end{equation}
Let $m=l+1$. We see from (\ref{e:cpvm13}) that
\[
|R_l(f,\mu,\{a_{l,i}\})-R_{l+1}(f,\mu,\{a_{l+1,i}\})|
  \leq C\cdot \sup_{k,i,s(i)}\{|\Lambda_h(D_kf;a_{m,s(i)},a_{l,i})|\},
\]
where $C$ is a constant.
So for any $\e>0$, there exists an integer $N>0$ such that
\[
|R_l(f,\mu,\{a_{l,i}\})-R_{l+1}(f,\mu,\{a_{l+1,i}\})|
  \leq \e
\]
for any $l\geq N$. Therefore $\{R_l(f,\mu,\{a_{l,i})\}_{l}$ is a
Cauchy sequence in ${\C}_\infty$. Thus the limit
$\lim\limits_{l\to \infty}
R_l(f,\mu,\{a_{l,i}\})$ exists.
\end{proof}
\par
From this lemma, an $h$-admissible measure $\mu$ on $S$ extends to a
linear functional on $C^h(S)$. Its extension on $C^h(S)$ is still
denoted by $\mu$.

\begin{cor}\label{cor:cpvm4}
An $h$-admissible measure $\mu$ satisfies
\[
\left|\int_{B(|\pi|^l)}(x-a)^jd\mu(x)\right|\leq C\cdot |\pi|^{l(j-h)}
\]
for all $j\geq 0$, and $a\in B(|\pi|^l)\subset S$ (where $C$ is a
constant not depending on $B(|\pi|^l)$).
\end{cor}
\begin{proof}
Apply the definition of Riemann sum for an $h$-admissible measure.
\end{proof}

\begin{prop}\label{prop:cpvm4}
The extended functional $\mu$ on $C^h(S)$ is continuous, so
$\mu\in (C^h(S))^*$.
\end{prop}
\begin{proof}
We need to prove $\mu$ is bounded.
Let $f\in C^h(S)$. In the proof of
Lemma \ref{lemma:cpvm4}, take $\e=||f||_h$, then there exists a
sufficiently large integer
$N$, such that
\begin{equation}\label{e:cpvm14}
|R_N(f,\mu,\{a_{N,i}\})-R_l(f,\mu,\{a_{l,i}\})|\leq ||f||_h
\end{equation}
for any $l\geq N$. Meanwhile,
\begin{equation*}
R_N(f,\mu,\{a_{N,i}\})=\sum_{i=0}^{r(N)}\sum_{j=0}^h
    D_jf(a_{N,i})\int_{B_i(|\pi|^N)}(x-a_{N,i})^j d\mu(x).
\end{equation*}
From $\max_{j, \, i}\{|D_j f(a_{N,i})|\}\le \max_{j}\{||\Phi_jf||_{\infty}\}
\leq C_1\cdot ||f||_h$ and the inequality (\ref{e:cpvm5}) in the definition of
$h$-admissible measures,  it is
clear that $|R_N(f,\mu,\{a_{N,i}\})|\leq C_2\cdot ||f||_h$. Therefore
$|R_l(f,\mu,\{a_{l,i}\})|\leq C_3\cdot ||f||_h$ from (\ref{e:cpvm14}) for
any $l\ge N$. This implies
\[
\left|\int_{S}f(x)d\mu(x)\right|\leq C_3\cdot ||f||_h,
\]
thus $\mu$ is bounded as a functional on $C^h(S)$.
\end{proof}

From all of the above discussion, we conclude
\begin{theorem}\label{thm:cpvm4}
The space of $h$-admissible measures on $S$ is dual to the ${\C}_\infty$-Banach
space $(C^h(S), ||\cdot||_h)$.
\end{theorem}

%===============================================================================================================

\section{Functions on Bruhat-Tits trees}\label{fbtt}

In this section, we'll recall Teitelbaum's results \cite{Te1991} on
the correspondence between the space of cusp forms and the space of
harmonic functions on the edges of the Bruhat-Tits tree $\mathcal
T$. Let $M$ be an abelian group and $G$ an arithmetic subgroup of
${\GL2({\bk})}$. Recall that ${\mathcal T}$ $=(V_{\mathcal T},
E_{\mathcal T})$ is the Bruhat-Tits tree associated to the field
${\bk}_\infty$ with an orientation chosen in Section \ref{btt}.
And $E_{\mathcal T}=E^{+}_{\mathcal T}\sqcup E^{-}_{\mathcal T}$ with
$E^{+}_{\mathcal T}$ the set of edges of positive orientation and
$E^{-}_{\mathcal T}$ the set of edges of negative orientation.

An $M$-valued function $\mathfrak c$ on the set $E_{\mathcal T}$ of
edges of $\mathcal T$ is called harmonic if the following two
conditions are satisfied:
\begin{itemize}
\item[(1)] $\sum_{e\mapsto v} {\mathfrak c}(e)=0$ for any vertex $v\in V_{\mathcal T}$, where the
sum is taken over all the oriented edges with final vertex $v$;
\item[(2)] for each $e\in E_{\mathcal T}$, we have
${\mathfrak c}(e)+{\mathfrak c}(\bar{e})=0$, where $\bar{e}$ denotes
the opposite edge of $e$.
\end{itemize}

For a positive integer $n$, let
\[
V(n)=\{F(X,Y)\in {\C}_\infty[X,Y]:\,\, \text{ $F$ is homogeneous of degree
$n-1$}\}.
\]
There is a natural action of ${\GL2}({\bk_\infty})$ on $V(n)$ given by
\[
(\gamma\cdot F)(X,Y)=F(aX+bY, cX+dY)
\]
for $\gamma=\left(\begin{array}{cc}a&b\\c&d\end{array}\right)\in {\GL2}({\bk}_\infty)$
and $F\in V(n)$. We assume ${\GL2}({\bk_\infty})$ acts trivially on the field ${\C}_\infty$,
and acts on $V(-n)=$ ${\rm Hom}_{{\C}_\infty}(V(n), {\C}_\infty)$ by diagonal, that is,
\[
(\gamma\cdot \lambda)(F)=\lambda(\gamma^{-1}\cdot F) \quad\text{ for $\gamma\in {\GL2}({\bk_\infty})$,
$\lambda\in V(-n)$, and $F\in V(n)$}.
\]

Suppose the arithmetic subgroup $G$ acts on
the Bruhat-Tits tree $\mathcal T$, with the action denoted by
``$\star$". The space $C_{\rm har}(G,n)$ of harmonic cocycles of
weight $n$ for $G$ is defined to be the space of $G$-invariant
$V(1-n)\otimes {\det}$ valued harmonic functions on $E_{\mathcal
T}$. Here the notation ``$\det$" represents the determinantal
representation of $G$, it twists the representation $V(1-n)$ (the
group $G$ acts on $V(1-n)\otimes {\det}$ by diagonal). Therefore
${\mathfrak c}\in C_{\rm har}(G,n)$ if and only if
\begin{equation}\label{e:fbtt1}
{\mathfrak c}(\gamma\star e)(F(X,Y))=\det(\gamma)\cdot {\mathfrak
c}(e)(F(aX+bY, cX+dY))
\end{equation}
for $e\in E_{\mathcal T}$, $F\in V(n-1)$, and $\gamma$ as above.
And the above condition on $\mathfrak c$ is also equivalent to
\begin{equation}\label{e:fbtt2}
{\mathfrak c}(\gamma\star e)(X^iY^{n-2-i})=\det(\gamma)\cdot
{\mathfrak c}(e)((aX+bY)^i(cX+dY)^{(n-2-i)}).
\end{equation}
\begin{rem}\label{rem:fbtt1}
As $G$ is a subgroup of ${\GL2}(\bk)\subset {\GL2}({\bk}_\infty)$, it has an induced
action on the Bruhat-Tits tree $\mathcal T$ from the natural action of
${\GL2}({\bk}_\infty)$ on ${\bk}_\infty\oplus{\bk}_\infty$.
In the subsequent content, we are also going to
use an action ``$\ast$" defined by
\begin{equation*}
\text{for $\gamma\in {\GL2}({\bk}_\infty)$ and $v\in {\bk}_\infty\oplus{\bk}_\infty$},
\quad\gamma\ast v:=(\gamma^{-1})^T\cdot v
\end{equation*}
where ``$\cdot$" denotes the natural action.
\end{rem}

To any oriented edge $e\in E_{\mathcal T}$, we associate with
the set $U(e)$ of ends of $\mathcal T$ which pass through $e$. By using the
bijection between the set of ends of $\mathcal T$ and
${\P1}({\bk}_\infty)$ chosen in Section \ref{btt}, we also denote
the corresponding subset of ${\P1}({\bk}_\infty)$ by $U(e)$. Then
$U(e)$ is an open compact subset of ${\P1}({\bk}_\infty)$.

\begin{example}\label{example:fbtt1}
For the edges $e=\Lambda_0 \Lambda_1$, $\Lambda_1\Lambda_0$, $\Lambda_1\Lambda_2$ on
the half line (\ref{e:2.5}), we have
\[
\begin{split}
&U(\Lambda_0\Lambda_1)=\{x\in {\bk}_\infty: \,\, |x|\ge |\pi|^{-1}\}\cup\{\infty\}=B_{\infty}(|\pi|^{1}),\\
&U(\Lambda_1\Lambda_0)=\{x\in {\bk}_\infty: \,\, |x|\le |\pi|^{0}\}=B_0(|\pi|^0)
 ={\A}_\infty,\\
&U(\Lambda_1\Lambda_2)=\{x\in {\bk}_\infty: \,\, |x|\ge |\pi|^{-2}\}\cup\{\infty\}=B_{\infty}(|\pi|^{2}).
\end{split}
\]
And more generally, for an integer $m\ge 0$
\[
\begin{split}
&U(\Lambda_{m}\Lambda_{m+1})=\{x\in {\bk}_\infty: \,\, |x|\ge |\pi|^{-(m+1)}\}\cup\{\infty\}=B_{\infty}(|\pi|^{m+1}),\\
&U(\Lambda_{m+1}\Lambda_{m})=\{x\in {\bk}_\infty: \,\, |x|\le
|\pi|^{-m}\}=B_0(|\pi|^{-m}).
\end{split}
\]
\end{example}

\begin{example}\label{example:fbtt2}
Consider the edges on the half line:
\begin{equation}\label{e:bbtt3}
\xymatrix@R=0pt{{\circ}\ar@{-}[r]&{\circ}\ar@{-}[r]&{\circ}\ar@{.}[r] &{\circ}\ar@{-}[r]&\ar@{.}[r]&\\
\Lambda_0 & \Lambda_{-1} & \Lambda_{-2} & \Lambda_{-m} &  }
\end{equation}
where $\Lambda_{-m}$ is the equivalence class of the lattice $=\pi^m
{\A_\infty}\oplus {\A_\infty}$ for each integer $m\ge 0$. Then
\[
\begin{split}
&U(\Lambda_{-m}\Lambda_{-(m+1)})=\{x\in {\bk}_\infty: \,\, |x|\le |\pi|^{m+1}\}=B_{0}(|\pi|^{m+1}),\\
&U(\Lambda_{-(m+1)}\Lambda_{-m})=\{x\in {\bk}_\infty: \,\, |x|\ge
|\pi|^{m}\}\cup\{\infty\}=B_{\infty}(|\pi|^{-m}).
\end{split}
\]
\end{example}

\begin{example}\label{example:fbtt3}
In general, let $M_j=[\pi^j {\bv}_1, x{\bv}_1+{\bv}_2]$ denote the
equivalence class of lattice generated by $\pi^j {\bv}_1$ and
$x{\bv}_1+{\bv}_2$ (the two vectors ${\bv_1}=(1,0)^T$,
${\bv_2}=(0,1)^T$ are the standard basis of
$V={\bk_\infty}\oplus{\bk_\infty}$) on the half line
(\ref{e:2.4.5}), where $x$ $=\sum_{j=j_0}^\infty c_j\pi^j$ $\in
{\bk}_\infty$ and $c_{j}\in {\Fq}$, then a little computation shows
that
\[
\begin{split}
U(M_j M_{j+1})=&\{y\in {\bk}_\infty: \,\, |y-x|\le |\pi|^{j+1}\}=B_x (|\pi|^{j+1}),\\
U(M_{j+1} M_j)=&\{y\in {\bk_\infty}: \,\, |y-x|=|\pi|^j\}\cup
\{y\in {\bk_\infty}: \,\, |y-x|=|\pi|^{j-1}\}\\
&\cup \cdots \cup \{y\in {\bk_\infty}: \,\, |y-x|=|\pi|^{j_0}\}\cup
\{y\in {\bk_\infty}: \,\, |y|\ge |\pi|^{j_0-1}\}\cup\{\infty\}\\
=&\{y\in {\bk}_\infty: \,\, |y-x|\ge |\pi|^j\}.
\end{split}
\]
\end{example}

Notice that the set of ends of $\mathcal T$ is in bijection with ${\P1}({\bk_\infty})$, as explained
in Section \ref{btt},  we conclude from the above examples that
\begin{prop}\label{prop:fbtt1}\begin{itemize}
\item[(1)] For any $j\in {\Z}$ and any $x\in {\bk_\infty}$, we have
\[
B_x(|\pi|^j)=U(M_{j-1} M_{j}), \qquad
B_{\infty}(|\pi|^j)=U(\Lambda_{j-1}\Lambda_{j}),
\]
where the notations are as in the Examples \ref{example:fbtt1}-\ref{example:fbtt3}.
\item[(2)] For any $e\in E_{\mathcal T}$, we have ${\P1}({\bk}_\infty)=U(e)\sqcup U(\bar{e})$. And $\infty\in U(e)$ if and only if $e\in E^{+}_{\mathcal T}$.
\end{itemize}
\end{prop}

For an integer $n\ge 2$, a harmonic cocycle ${\mathfrak c}\in
C_{\har}(G,n)$ can be associated with a $\mu_{\mathfrak c}\in
(P^{(n-2)})^*={\rm Hom}_{\C_\infty}(P^{(n-2)}, {\C}_\infty)$ (the
space $P^{(n)}$ is defined in Definition \ref{defn:cpvm1}) as
follows. Given an $e\in E_{\mathcal T}$ and an integer $i$ with
$0\le i\le n-2$, we define
\begin{equation}\label{e:fbtt4}
\int_{U(e)} x^i d\mu_{\mathfrak c}(x)={\mathfrak
c}(e)(X^iY^{n-2-i}),
\end{equation}
and extend to $P^{(n-2)}$ by linearity. Proposition \ref{prop:fbtt1}
assures $\mu_{\mathfrak c}\in (P^{(n-2)})^*$.

\begin{lemma}[\cite{Te1991}]\label{lemma:fbtt1}
For ${\mathfrak c}\in C_{\har}(G,n)$,
$\gamma=\left(\begin{array}{cc}a&b\\c&d\end{array}\right)\in G$,
$e\in E_{\mathcal T}$, and $f$ a polynomial of degree at most
$n-2$, the following holds.
\begin{itemize}
\item[(1)] we have
\begin{equation}\label{e:fbtt5}
\int_{U(\gamma\star e)}f(x) d\mu_{\mathfrak
c}(x)=\int_{U(e)}\det(\gamma)\cdot f(\gamma
x)(cx+d)^{n-2}d\mu_{\mathfrak c}(x),
\end{equation}
where the action of $G\subset {\GL2}({\bk_\infty})$ on ${\P1}(\bk_\infty)$ is given by
$\gamma x=(ax+b)/(cx+d)$.
\item[(2)] \begin{equation}\label{e:fbtt5.5}\int_{{\P1}({\bk}_\infty)} f(x) d\mu_{\mathfrak c}(x)=0.\end{equation}
\item[(3)] There exists a constant $C>0$ such that for all $0\le i\le n-2$,
\begin{equation}\label{e:fbtt6}
\left|\int_{U(e)} (x-a)^i d\mu_{\mathfrak c}(x)\right|\le
C\rho(e)^{i-(n-2)/2},
\end{equation}
where $e\in E^{-}_{\mathcal T}$, $a\in U(e)$, and $\rho(e)=\sup_{x,\,y\in U(e)}|x-y|$ is
the diameter (or radius, the same in ultra-metric analysis) of $U(e)$.
\end{itemize}
\end{lemma}
\begin{proof}
See \cite{Te1991}.
\end{proof}

Let $h$ be the smallest integer greater than or equal to $(n-2)/2$.

By Proposition \ref{prop:fbtt1}, every closed ball is of the form $U(e)$ for some
edge $e$ of the tree $\mathcal T$, therefore the inequality (\ref{e:fbtt6}) implies
that the measure $\mu_{\mathfrak c}$ is $h$-admissible on any
open compact subset $S$ of ${\bk}_\infty$. Hence the functions of $C^h(S)$ are
integrable against $\mu_{\mathfrak c}$ by Theorem \ref{thm:cpvm4}. For
$B_\infty(|\pi|^m)=U(\Lambda_{m-1}\Lambda_{m})$ $=\{x\in{\bk}_\infty:\,\,
|x|\ge |\pi|^{-m}\}\cup\{\infty\}$ where $m$ is an integer, we can
decompose it as
\begin{align*}
&B_\infty(|\pi|^m)\\
=&\,\left(\bigsqcup_{m\le s\le l}\{x\in{\bk}_\infty:\,\,
|x|=|\pi|^{-s}\}\right)\bigsqcup
\left(\{x\in{\bk}_\infty:\,\,|x|\ge |\pi|^{-(l+1)}\}\cup \{\infty\}\right)\\
=&\,\left(\bigsqcup_{c\in{\Fq}^{\ast}}\bigsqcup_{m\le s\le l}
B_{c\pi^{-s}}(|\pi|^{-s+1})\right)
\bigsqcup B_\infty(|\pi|^{l+1})
\end{align*}
for $l\ge m$ a large integer, and define for an integer $i> 0$
\begin{equation}\label{e:fbtt6.1}
\int_{B_\infty(|\pi|^m)} x^{-i}d\mu_{\mathfrak c}(x)
=\lim_{l\to \infty}\sum_{c\in{\Fq}^*}\sum_{m\le s\le l}
\int_{B_{c\pi^{-s}}(|\pi|^{-s+1})}x^{-i}d\mu_{\mathfrak c}(x).
\end{equation}
The limit of (\ref{e:fbtt6.1}) exists because
\begin{align*}
\left|\int_{B_{c\pi^{-l}}(|\pi|^{-l+1})}x^{-i}d\mu_{\mathfrak c}(x)\right|
&\le \left|\int_{B_{c\pi^{-l}}(|\pi|^{-l+1})}
\left(\dfrac{1}{c\pi^{-l}}\right)^i
\left(\dfrac{1}{1+(x-c\pi^{-l})/(c\pi^{-l})}\right)^i d\mu_{\mathfrak c}(x)\right|\\
&\,\le C\cdot |\pi|^{l(i+(n-2)/2)}
\end{align*}
where $C$ is a constant independent of the closed balls contained
in ${\P1}(\bk_\infty)$, and we have the estimate
\begin{equation}\label{e:fbtt6.2}
\left|\int_{B_\infty(|\pi|^m)} x^{-i}d\mu_{\mathfrak c}(x)\right|
\le C\cdot \left(|\pi|^m\right)^{(i+(n-2)/2)} \qquad\text{ for $i\ge 0$}
\end{equation}
where $C$ is a constant independent of the closed balls
contained in ${\P1}(\bk_\infty)$.

In Lemma \ref{lemma:fbtt1}, we take $e=\Lambda_{m-1}\Lambda_m$. Then
$\bar{e}=\Lambda_m\Lambda_{m-1}$ is the edge with opposite
orientation of that of $e$, and $U(e)=B_{\infty}(|\pi|^m)$,
$U(\bar{e})=B_0(|\pi|^{-(m-1)})$.  Therefore ${\mathfrak
c}(e)+{\mathfrak c}(\bar{e})=0$ implies
\begin{equation}\label{e:fbtt7}
\left|\int_{B_{\infty}(|\pi|^m)}x^i d\mu_{\mathfrak c}(x)\right|
 = \left|\int_{B_0(|\pi|^{-(m-1)})}
x^{i}d\mu_{\mathfrak c}(x)\right|
 \le C\cdot \left( |\pi|^m
\right)^{-i+(n-2)/2}.
\end{equation}
Combine (\ref{e:fbtt6.2}) and (\ref{e:fbtt7}), we have
\begin{prop}[\cite{Te1991}]\label{prop:fbtt1.1}
There exists a constant $C$ such that for $m\in {\mathbb Z}$ and $-\infty\le i\le n-2$
\begin{equation}\label{e:fbtt8}
\left|\int_{B_{\infty}(|\pi|^m)}x^i d\mu_{\mathfrak c}(x)\right|
\le C\cdot \left(|\pi|^m\right)^{-i+(n-2)/2}.
\end{equation}
\end{prop}

Therefore by recalling Theorem \ref{thm:cpvm4} of Section \ref{cpvm}, we have
\begin{cor}\label{cor:fbtt1}
The map
\[
\begin{array}{rcl}
C_{\har}(G,n)&\to &(C^h({\P1}(\bk_\infty))^*\\
{\mathfrak c}&\mapsto &\mu_{\mathfrak c}
\end{array}
\]
injects $C_{\har}(G,n)$ into $(C^h({\P1}(\bk_\infty))^*$ as the
subspace of $h$-admissible measures satisfying the equations
(\ref{e:fbtt5}) and (\ref{e:fbtt5.5}) for a polynomial $f$ of degree
at most $n-2$.
\end{cor}
\begin{proof}
As every closed ball of ${\P1}(\bk_\infty)$ is of the form $U(e)$ for some $e\in E_{\mathcal T}$
by Proposition \ref{prop:fbtt1}, this is direct from the estimates (\ref{e:fbtt6}) and (\ref{e:fbtt8}).
\end{proof}

\begin{rem}\label{rem:fbtt2}
Due to the inequality (\ref{e:fbtt7}), not only we can integrate
$C^h$ functions on ${\P1}({\bk}_\infty)$ against $\mu_{\mathfrak
c}$, but also we can integrate those functions with a pole at
$\infty$ of order at most $n-2$. More generally, if the integral
$\int_{U(e)}f(x)d\mu_{\mathfrak c}(x)$ is defined for a $C^h$
function $f$ and an edge $e$ of $\mathcal T$, then we can also
formally define
\begin{equation}\label{e:fbtt8.1}
\int_{U(\bar{e})}f(x)d\mu_{\mathfrak c}(x):=
-\int_{U(e)}f(x)d\mu_{\mathfrak c}(x).
\end{equation}
We'll always apply this convention in the subsequent discussion.
Finally, we notice that although $\mu_{\mathfrak c}$ is constructed
as an element of $(P^{(n-2)})^*$ by definition, but we see that
$\mu_{\mathfrak c}$ is determined by its values on $P^{(h)}$.
\end{rem}

\begin{cor}\label{cor:fbtt3}
Equation (\ref{e:fbtt5}) induces an action of $G$ on $C_{\rm har}(G,n)$, when viewed as
a subspace of $(C^h({\P1}(\bk_\infty)))^*$:
\begin{equation}\label{e:fbttmeq}
\gamma\cdot d\mu_{\mathfrak c}(x)={\det}(\gamma)
(cx+d)^{n-2}d\mu_{\mathfrak c}(x) \quad\text{for ${\mathfrak c}\in
C_{\rm har}(G,n)$ and $\gamma
=\left(\begin{array}{cc}a&b\\c&d\end{array}\right)$}.
\end{equation}
\end{cor}

There is a natural relation between the Bruhat-Titts tree $\mathcal T$ and the Drinfeld's upper half
plane $\Omega$, which we will describe below, more exposition about it can be found in \cite{Ge1996}.
The Bruhat-Tits tree $\mathcal T$ is in fact a pair of sets $V_{\mathcal T}$ and $E_{\mathcal T}$
with two maps $o,t: E_{\mathcal T}\to V_{\mathcal T}$ designating the origin $o(e)$ and the terminal $t(e)$
of an element $e$ of $E_{\mathcal T}$, and an orientation given in Section \ref{btt}.
Let ${\mathcal T}({\mathbb R})$ be the the realization of $\mathcal T$ consisting of a unit interval for
every edge (without considering orientation) of $\mathcal T$ which is glued at the extremities according to
the incidence relations of $\mathcal T$. The set of vertices of ${\mathcal T}({\mathbb R})$ is
${\mathcal T}(\Z)$, which corresponds to $V_{\mathcal T}$.

An edge $e$ of ${\mathcal T}({\mathbb R})$ corresponds to some edge $[L][L']$, where $L$ and $L'$ are ${\A}_\infty$-lattices
of $V={\bk_\infty}\oplus {\bk_\infty}$ satisfying $\pi L'\subset L\subset L'$ (which is the same
as saying that $d([L], [L'])=1$ in Section \ref{btt}). Then every point $P$ of $e$ can be formally expressed
as
\begin{equation}\label{e:fbtt9}
P=(1-t)[L]+t[L'] \quad\text{ for $0\le t\le 1$.}
\end{equation}
We have defined an action $\gamma\cdot [L]$ for $\gamma\in{\GL2}(\bk_\infty)$
and $L$ an ${\A}_\infty$-lattice in Section \ref{btt}, which is induced
from the ordinary action of ${\GL2}(\bk_\infty)$ on $V={\bk}_\infty\oplus{\bk}_\infty$. There is
another action
\begin{equation}\label{e:staraction}
\begin{array}{rcrcl}
{\GL2}({\bk}_\infty)&\times &V_{\mathcal T}&\rightarrow &V_{\mathcal T}\\
(\gamma&, &[L]) &\mapsto &\gamma\ast [L]:=({\gamma}^T)^{-1}\cdot [L]=[({\gamma}^T)^{-1}\cdot L]
\end{array}
\end{equation}
which induces an action ``$\ast$" of $\GL2({\bk_\infty})$ on ${\mathcal T}({\mathbb R})$.

To any lattice $L$ of $V$, we associate a norm $\nu_{L}$ by
\begin{equation}\label{e:fbtt10}
\nu_L(v) = \inf\{|x|:\,\, x\in {\bk_\infty}, v\in xL\}, \quad \text{ for any $v\in V$ }.
\end{equation}
This is equivalent to saying that $L$ is the unit ball
with respect to $\nu_L$. And for the point $P\in {\mathcal T}(\mathbb R)$ in (\ref{e:fbtt9}),
we define a norm $\nu_P$ by
\begin{equation}\label{e:fbtt11}
\nu_P(v)=\max\{\nu_L(v), q^t\nu_{L'}(v)\}, \quad \text{ for $v\in V$ }.
\end{equation}

Two norms $\nu_1$ and $\nu_2$ on $V$ are said to be similar if there
exists a real number $C>0$ such that $\nu_2=C \nu_1$. Similarity of norms on $V$
is an equivalence relation. The equivalence class of $\nu$ is denoted by $[\nu]$. It is easy to
see that equivalent lattices correspond to similar norms, and we have the well-defined
notions $[\nu_{[L]}]=[\nu_L]$ and $[\nu_P]$ according to (\ref{e:fbtt9})--(\ref{e:fbtt11}).

\begin{example}\label{example:fbtt4}
If we denote by $L={\A}_\infty\oplus {\A}_\infty$, and $L'=\pi{\A}_\infty\oplus {\A}_\infty$, then
for $v=(a,b)^T\in V={\bk}_\infty\oplus{\bk}_\infty$,
\[
\nu_L(v)=\max\{|a|, |b|\}, \qquad \nu_{L'}(v)=\max\{q\cdot|a|,
|b|\}.
\]
\end{example}

The group ${\GL2}({\bk}_\infty)$ acts on the space of similarity classes of norms on $V$ by
\begin{equation}\label{e:fbtt12}
(\gamma\cdot\nu)(v)=\nu(\gamma^{T}\cdot v)
\end{equation}
for $\gamma\in {\GL2}({\bk}_\infty)$, $v\in V$, and $\nu$ a norm on
$V$ (this induces an action $\gamma\cdot [\nu]$), where ${\gamma}^T$ denotes
the transpose of the matrix $\gamma$.

\begin{theorem}[Goldman-Iwahori, see \cite{Ge1996}]\label{thm:fbtt1} The association
of a point $P$ in ${\mathcal T}(\mathbb R)$ with $[\nu_P]$ establishes a canonical
${\GL2}(\bk_\infty)$-equivariant bijection between ${\mathcal T}(\mathbb R)$ (with respect to the ``$\ast$" action defined by the equation (\ref{e:staraction})) and the
space of similarity classes of norms on $V$.
\end{theorem}

To any $z\in \Omega$ we also associate $[\nu_z]$ of the norm $\nu_z$ on $V$ by
\begin{equation}\label{e:fbtt13}
\nu_z((u,v)):=|uz+v| \quad\text{ for $u,v\in {\bk}_\infty$.}
\end{equation}
Due to Theorem \ref{thm:fbtt1}, we can identify ${\mathcal T}(\mathbb R)$ with the space of
similarity classes of norms on $V$. Therefore (\ref{e:fbtt13}) gives the ``building map"
\[
\begin{array}{rcl}
\lambda: \Omega &\to &{\mathcal T}(\mathbb R)\\
            z   &\mapsto &[\nu_z].
\end{array}
\]
The image of $\lambda$ is ${\mathcal T}(\mathbb Q)$.
\begin{prop}[\cite{Ge1996}]\label{prop:fbtt2}
The building map $\lambda$ is ${\GL2}(\bk_\infty)$-equivariant, that is, it satisfies
\[
\lambda(\gamma z)=[\nu_{\gamma z}]=\gamma\cdot [\nu_z]
\]
for any $z\in \Omega$ and any $\gamma\in {\GL2}(\bk_\infty)$.
\end{prop}
\begin{proof}
This is a direct verification. Notice that the notion of the action
of ${\GL2}(\bk_\infty)$ given by (\ref{e:fbtt12}) is a little
different from \cite{Ge1996}, the group ${\GL2}(\bk_\infty)$ always
acts on the left in this paper.
\end{proof}

\begin{example}\label{example:fbtt5}
Let $\Lambda_n$ be the equivalence class of the ${\A}_\infty$-lattice
$L_n=\pi^{-n}{\A}_\infty\oplus {\A}_\infty$ as in Section \ref{btt}.
\begin{itemize}
\item[(1)] ${\lambda}^{-1}(\Lambda_{n})=\left\{z\in \Omega:\,\,
|z-c\pi^n| = |\pi^n| \text{ for all $c\in {\F}_q$ }\right\}$.
\item[(2)] Let $\tilde{e}_n=\Lambda_n\Lambda_{n+1}\in {\mathcal T}(\mathbb R)$ be the edge with the
two end points $\Lambda_n$ and $\Lambda_{n+1}$, and $e_n$ the
edge with the two end points removed. Then
\[
\begin{split}
&{\lambda}^{-1}(e_n)=\{z\in \Omega:\,\, |\pi^{n+1}|<|z|<|\pi^n|\},\\
&{\lambda}^{-1}(\tilde{e}_n)=\left\{z\in \Omega:\,\,
\begin{array}{ll}
&|\pi^{n+1}|\le |z|\le |\pi^n|,\\
&|z-c\pi^n|\ge |\pi^n|, |z-c\pi^{n+1}|\ge |\pi^{n+1}| \text{ for all
$c\in {\F}_q^*$ }
\end{array}
\right\}.
\end{split}
\]
\end{itemize}
These relations can be verified directly from the definition of the
building map $\lambda$. A natural way to understand them is by considering
some analytic reduction of $\Omega$ to a locally finite scheme over ${\Fq}$,
which gives rise to the Bruhat-Tits tree $\mathcal T$, as in \cite{Ge1996}.
\end{example}

Now we identify $\mathcal T$ with its realization ${\mathcal T}(\mathbb R)$ and let $e_n$ be as in
Example \ref{example:fbtt5}. Let $e\in E_{\mathcal T}$ with the two endpoints removed and
$D(e)={\lambda}^{-1}(e)$. Then
\[
D:=D(e_0)=\{z\in \Omega:\,\, |\pi|<|z|<1\}
\]
is a fundamental region of $\Omega$. Since $\tilde{e}_0=\Lambda_0\Lambda_1$ is a
fundamental domain of the action of ${\GL2}(\bk_\infty)^{+}$ on $\mathcal T$ as explained in
Section \ref{btt}, we have
\begin{equation}\label{e:edgerelation}
e={\gamma}\ast e_0=({\gamma}^{-1})^T\cdot e_0
\end{equation}
and there is an isomorphism
\begin{equation}\label{e:parametrization}
\begin{array}{rcl}
  D& \rightarrow &D(e)\\
   v&\mapsto &z=\gamma v
\end{array}
\end{equation}
given by an element $\gamma\in {\GL2}(\bk_\infty)^{+}$, from Theorem
\ref{thm:fbtt1} and Proposition \ref{prop:fbtt2}. Then any rigid
differential form $f dz$ on $D(e)$ can be expressed by
\begin{equation}\label{e:fbtt15}
\sum_{i\in {\Z}} a_i v^i dv.
\end{equation}
We have the following definitions with respect to rigid differential
forms (see \cite{Te1991}).
\begin{itemize}
\item Define ${\rm Res}_e fdz:=a_{-1}$, where $a_{-1}$ is the coefficient of $v^{-1}dv$ in the equation (\ref{e:fbtt15}).
Here $e\in E_{\mathcal T}$ is an oriented edge which is related to $e_0$ by (\ref{e:edgerelation}).
\item Suppose $f$ is a modular form of weight $n$ for an arithmetic subgroup $G\subset {\GL2}(\bk)$. Then
define a harmonic cocycle ${\rm Res}(f)$ of weight $n$ by
\begin{equation}\label{e:residue}
{\rm Res}(f)(e)(X^iY^{n-2-i}):={\rm Res}_e \,z^if(z)dz.
\end{equation}
\end{itemize}
In the above definitions and the discussions throughout the rest of
this paper, we will not distinguish an edge with the end points or
without the end points, since the results will be the same. It can
be verified that ${\rm Res}(f)$ is harmonic and $G$-equivariant
(note that we need to use the action ``$\ast$" of $G$ on ${\mathcal T}$
defined in (\ref{e:staraction}) for ${\rm Res}(f)$ to be
$G$-equivariant).

\begin{theorem}[Teitelbaum, \cite{Te1991}]\label{thm:fbtt2}
Let ${\mathfrak c}\in C_{\rm har}(G,n)$ be a harmonic cocycle of
weight $n\ge 2$ for the arithmetic group $G$, and let
$\mu_{\mathfrak c}$ be the associated measure. Define $f$ by the
integral
\[
f(z)=\int_{{\mathbb P}^1(\bk_\infty)} \frac{1}{z-x} d\mu_{\mathfrak
c}(x).
\]
Then $f$ is a rigid analytic cusp form for $G$ of weight $n$ and
${\rm Res}(f)={\mathfrak c}$. Moreover, let $e\in E^-_{\mathcal T}$
(thus $\infty\not\in U(e)$) and $a$ be a center of $U(e)$. Then for
all $m\in{\Z}$, we have
\begin{equation}\label{e:fbtt21}
{\rm Res}_e(z-a)^mf(z)dz=\int_{U(e)}(x-a)^md\mu_{\mathfrak c}(x).
\end{equation}
\end{theorem}
\begin{rem}\label{rem:fbtt3}
In the equation (\ref{e:fbtt21}) of Theorem \ref{thm:fbtt2}, the
integral $\int_{U(e)}(x-a)^md\mu_{\mathfrak c}(x)$ is certainly well
defined for all $m\ge 0$, because $\infty\not\in U(e)$ and so
$(x-a)^m$ is an analytic function on $U(e)$ when $m\ge 0$. For
$m<0$, the function $(x-a)^m$ has a pole at the point $a\in U(e)$,
therefore the integral can not be defined directly. But the function
$(x-a)^m$ is analytic on $U(\bar{e})$, thus we set
\begin{equation*}
\int_{U(e)}(x-a)^md\mu_{\mathfrak
c}(x):=-\int_{U(\bar{e})}(x-a)^md\mu_{\mathfrak c}(x)
\end{equation*}
as Remark \ref{rem:fbtt2} explains.
\end{rem}
\begin{theorem}[Teitelbaum, \cite{Te1991}]\label{thm:fbtt3}
Let $G$ be an arithmetic subgroup of ${\GL2}(\bk)$, and let $S_n(G)$
denote the space of rigid cusp forms of weitht $n\ge 2$ for $G$.
Then the maps
\[
{\rm Res}: S_n(G)\rightarrow C_{\rm har}(G,n)
\]
and
\[
\begin{array}{rcl}
\iota: C_{\rm har}(G,n)&\rightarrow &S_n(G)\\
      {\mathfrak c}&\mapsto &\int_{{\mathbb P}^1(\bk_\infty)}\frac{1}{z-x}
      d\mu_{\mathfrak c}(x)
\end{array}
\]
are mutually inverse isomorphisms.
\end{theorem}

The set ${\mathcal U}_1=\{x\in{\bk}_\infty:\,\, |x-1|<1\}$ of $1$-units of ${\bk}_\infty$ can be given as
\begin{equation*}
{\mathcal U}_1=U(M_0M_1)
\end{equation*}
where $M_0M_1$ is the edge with initial vertex $M_0=[\bv_1,\bv_2]$ and terminal vertex
$M_1=[\pi \bv_1, \bv_1+\bv_2]$, as explained in Example \ref{example:fbtt3}. For $G=\Gamma={\GL2}(\A)$
and $f\in S_n(\Gamma)$, an analogue of $L$-functions of classical cusp forms is $L_f(s):{\Z}_p\to {\C}_\infty$
given by (see \cite{Go1992})
\begin{equation}\label{e:lfunction}
L_f(s)=\int_{{\mathcal
U}_1}x^s\dfrac{d\mu_f(x)}{x}=\int_{U(M_0M_1)}x^{s-1}d\mu_f(x)
\end{equation}
where $\mu_f$ is the measure associated to the cocycle
${\iota}^{-1}(f)$ under the correspondence of Theorem
\ref{thm:fbtt3}.  In general, let $G$ be an arithmetic subgroup of
${\GL2}(\bk)$ and $f\in S_n(G)$ with $n\ge 2$. Then for an edge $e$
of $\mathcal T$, we define
\begin{equation}\label{e:fbtt22}
L(f;e;j)=\int_{U(e)} x^{j-1} d\mu_f(x)\quad\text{ for $j\in{\Z}$}
\end{equation}
In the case when the edge $e=M_0M_1$, the function $L(f;e;j): {\Z}\to {\C}_\infty$
in the variable $j$ can be extended to the function $L(f;e;s)$ for $s\in{\Z}_p$ by
continuity, which is the $L$-function $L_f(s)$ defined in (\ref{e:lfunction}).

%=============================================================================================================================

\section{The measure associated to the Drinfeld
discriminant}\label{mdd}

Example \ref{example:3.4} explains how we get the Drinfeld discriminant $\Delta$.
In this section, we will compute $C_{\rm har}(\Gamma,q^2-1)$, hence will get
the measure associated to the Drinfeld discriminant $\Delta$.

\begin{theorem}[Goss, \cite{Go1980a}]\label{thm:mdd1}
The dimension of the space of cusp forms of weight $m=m_0(q-1)$ and type $0$ of the arithmetic subgroup
$\Gamma$ is $[m_0/(q+1)]$. The dimension of the space of all Drinfeld modular forms of weight
$m=m_0(q-1)$ and type $0$ of the arithmetic subgroup $\Gamma$ is $[m_0/(q+1)]+1$.
\end{theorem}

By the above theorem and Theorem \ref{thm:fbtt3}, we see that $C_{\rm har}(\Gamma, q^2-1)$,
as well as $S_{q^2-1}(\Gamma)$, is a one dimensional vector space over ${\C}_\infty$. The space $S_{q^2-1}(\Gamma)$
of cusp forms is generated by the Drinfeld discriminant $\Delta$.
Let ${\mathfrak c}_{\Delta}={\rm Res}(\Delta)$ be the harmonic cocycle in Theorem \ref{thm:fbtt3}.
Then ${\mathfrak c}_{\Delta}$ is a basis of $C_{\rm har}(\Gamma, q^2-1)$.
We can calculate ${\mathfrak c}_{\Delta}$ by
\[
{\mathfrak c}_{\Delta}(e)(X^i Y^{q^2-3-i})={\rm Res}_e z^i\Delta(z) dz
\]
for $0\le i\le q^2-3$ from the
definition (\ref{e:residue}) of harmonic cocycles coming from modular forms. Since the
harmonic cocycle associated to $\Delta$ is $\Gamma$-invariant and the
half line (\ref{e:2.1}) consisting of $\Lambda_m$ ($m\ge 0$) is the fundamental domain of
$\Gamma\backslash {\mathcal T}$, the values ${\rm Res}_{e_m} z^i\Delta(z) dz$
(for $0\le i\le q^2-3$) determine ${\mathfrak c}_{\Delta}$, where $e_m$ is
the oriented edge $\Lambda_m\Lambda_{m+1}$.

At first, we are going to calculate ${\rm Res}_{e_0} z^i\Delta(z)dz$,
the residue of $z^i\Delta(z)dz$ over the region $D=\lambda^{-1}(e_0)$
$=\{z\in \Omega: |\pi|<|z|<1\}$. As $\Delta(z)=(T^{q^2}-T)E_{q^2-1}(z)+(T^q-T)^q E_{q-1}(z)^{q+1}$,
we will expand $(E_{q-1}(z))^{q+1}$ and $E_{q^2-1}(z)$ on the region $D$ in the following.
\begin{align}
E_{q-1}(z)&=\sum\limits_{(0,0)\ne (c,d)\in {\A}^2} \dfrac{1}{(cz+d)^{q-1}}\notag\\
&=\sum_1\dfrac{1}{(cz+d)^{q-1}} + \sum_2\dfrac{1}{(cz+d)^{q-1}}\label{e:mdd1}
\end{align}
where $\sum\limits_1$ is taken over all $(c,d)\in{\A}^2$ such that $d\ne 0$ and $\deg(d)\ge \deg(c)$,
and $\sum\limits_2$ is taken over all $(c,d)\in{\A}^2$ such that $c\ne 0$ and
$\deg(c)\ge \deg(d)+1$. Here we set $\deg(0)=-\infty$ by convention. We'll compute
$\sum\limits_1$ and $\sum\limits_2$ separately.
\begin{align}
&S_1:=\sum_1\dfrac{1}{(cz+d)^{q-1}}\notag\\ %=\sum\limits_{\substack{(c,d)\in{\A}^2,\,d\ne 0 \\\deg(d)\ge \deg(c)}}
=&\sum_1 \left(\dfrac{1}{1+\frac{c}{d}z}\right)^{q-1}\cdot \dfrac{1}{d^{q-1}}\notag\\
=&\sum_1 \dfrac{1}{d^{q-1}}\left(\sum\limits_{i\ge 0}\left(\dfrac{c}{d}\right)^i z^i\right)^{q-1}\notag\\
=&\sum_1 \dfrac{1}{d^{q-1}}\left(1-\dfrac{c}{d}z+\left(\dfrac{c}{d}\right)^q z^q-\left(\dfrac{c}{d}\right)^{q+1} z^{q+1}
+\left(\dfrac{c}{d}\right)^{2q} z^{2q}-\left(\dfrac{c}{d}\right)^{2q+1} z^{2q+1}+\cdots\right)\label{e:mdd2}
\end{align}
where we have applied the identity in the following lemma.
\begin{lemma}\label{lem:mdd1}
For $B\in {\C}_\infty$, we have
\[
\left(\sum\limits_{i\ge 0} B^i z^i\right)^{q-1}=1-Bz+B^q z^q-B^{q+1}z^{q+1}
+B^{2q}z^{2q}-B^{2q+1}z^{2q+1}+\cdots
\]
\end{lemma}
\begin{proof}
We directly verify that
\[
\left(\sum\limits_{i\ge 0} B^i z^i\right)^q
=\left(\sum\limits_{i\ge 0} B^i z^i\right)\left(
1-Bz+B^qz^q-B^{q+1}z^{q+1}+B^{2q}z^{2q}-B^{2q+1}z^{2q+1}+\cdots\right)
\]
and get the identity in the lemma.
\end{proof}
And also for the sum $\sum\limits_2$\,,
\begin{align}
&S_2:=\sum_2\dfrac{1}{(cz+d)^{q-1}}\notag\\
=&\sum_2 \dfrac{1}{(cz)^{q-1}}\left(\dfrac{1}{1-\frac{d}{cz}}\right)^{q-1}\notag\\
=&\sum_2\dfrac{1}{(cz)^{q-1}}\left(\sum\limits_{i\ge 0}\left(\dfrac{d}{c}\right)^i
\dfrac{1}{z^i}\right)^{q-1}\notag\\
=&\sum_2\dfrac{1}{(cz)^{q-1}}\left(1-\dfrac{d}{c}\dfrac{1}{z}+\left(\dfrac{d}{c}\right)^q\dfrac{1}{z^q}
-\left(\dfrac{d}{c}\right)^{q+1}\dfrac{1}{z^{q+1}}+
\left(\dfrac{d}{c}\right)^{2q}\dfrac{1}{z^{2q}}-\left(\dfrac{d}{c}\right)^{2q+1}\dfrac{1}{z^{2q+1}}+\cdots\right)
\label{e:mdd3}
%=&\sum_2 \bigg(\dfrac{1}{c^{q-1}}\dfrac{1}{z^{q-1}} -\dfrac{1}{c^{q-1}}\dfrac{d}{c}\dfrac{1}{z^q}
%   +\dfrac{1}{c^{q-1}}\left(\dfrac{d}{c}\right)^q\dfrac{1}{z^{2q-1}}-
%   \dfrac{1}{c^{q-1}}\left(\dfrac{d}{c}\right)^{q+1}\dfrac{1}{z^{2q}}\notag\\
% &\qquad\qquad\qquad\qquad\qquad\qquad\quad\quad +\dfrac{1}{c^{q-1}}\left(\dfrac{d}{c}\right)^{2q}\dfrac{1}{z^{3q-1}}
%   -\dfrac{1}{c^{q-1}}\left(\dfrac{d}{c}\right)^{2q+1}\dfrac{1}{z^{3q}}+\cdots\bigg).\label{e:mdd3}
\end{align}
In the summation (\ref{e:mdd2}), we notice that
\begin{itemize}
\item the condition $\deg(d)\ge \deg(c)$ for $\sum\limits_1$ ensures that $d\ne 0$,
since otherwise $\deg(c)\le \deg(d)=-\infty$ implies $c=0$;
\item the summation can be taken term by term inside the parenthesis; and we have
\[
\sum\limits_1 \dfrac{1}{d^{q-1}}\left(\dfrac{c}{d}\right)^n z^n
=\sum\limits_{\substack{(c,\,d)\in {\A}^2\\ \deg(c)\le \deg(d)}}\dfrac{1}{d^{q-1}}\left(\dfrac{c}{d}\right)^n z^n
= 0 \text{\,\, if $(q-1)\not |\,n$, }
\]
since we can pull out the leading coefficients of $d\in {\A}={\Fq}[T]$ in the
terms of a fixed degree of $d$ and sum these
terms up.
\end{itemize}
Therefore the summation $S_1$ in (\ref{e:mdd2}) equals
\begin{align}
\sum\limits_{\substack{(c,\,d)\in {\A}^2\\\deg(d)\ge \deg(c)}}& \dfrac{1}{d^{q-1}}\bigg(1+
  \left(\dfrac{c}{d}\right)^{(q-1)q}z^{(q-1)q}+\left(\dfrac{c}{d}\right)^{2(q-1)q}z^{2(q-1)q}+\cdots \notag\\
  \qquad\qquad &-\left(\dfrac{c}{d}\right)^{(q-1)^2}z^{(q-1)^2}-\left(\dfrac{c}{d}\right)^{(2q-1)(q-1)}z^{(2q-1)(q-1)}
  -\left(\dfrac{c}{d}\right)^{(3q-1)(q-1)}z^{(3q-1)(q-1)}-\cdots\bigg) \label{e:mdd4}
\end{align}
where the terms with negative sign occur because we need to solve out the equations
\[
mq+1=n(q-1)
\]
as
\[
\begin{cases}
m=-1+t(q-1),\\
n=-1+tq,
\end{cases}
\text{ for $t\ge 1$.}
\]
In the same way, we notice that in summation (\ref{e:mdd3})
\begin{itemize}
\item $c$ can not be equal to $0$, otherwise $\deg(d)\le \deg(c)-1=-\infty$ implies
$d=0$,
\item we can take the summation term by term inside the parenthesis, and
\[
\sum\limits_2 \dfrac{1}{c^{q-1}}\left(\dfrac{d}{c}\right)^n \dfrac{1}{z^{n+q-1}}
=\sum\limits_{\substack{(c,\,d)\in {\A}^2\\ \deg(c)\ge \deg(d)+1}}
\dfrac{1}{c^{q-1}}\left(\dfrac{d}{c}\right)^n \dfrac{1}{z^{n+q-1}}
= 0 \text{\,\, if $(q-1)\not |\,n$, }
\]
for example, the first term is
\[
\sum\limits_{\substack{(c,\,d)\in{\A}^2\\\deg(c)\ge \deg(d)+1}}\dfrac{1}{c^{q-1}z^{q-1}}
=\sum\limits_{\substack{(c,\,d)\in{\A}^2\\\deg(c)=0,\,d=0}}\dfrac{1}{c^{q-1}z^{q-1}}
=-\dfrac{1}{z^{q-1}}.
\]
\end{itemize}
Therefore the summation $S_2$ in (\ref{e:mdd3}) equals
\begin{align}
\sum\limits_{\substack{(c,\,d)\in {\A}^2\\\deg(c)\ge \deg(d)+1}}& \dfrac{1}{c^{q-1}z^{q-1}}\bigg(1+
  \left(\dfrac{d}{c}\right)^{(q-1)q}\dfrac{1}{z^{(q-1)q}}+\left(\dfrac{d}{c}\right)^{2(q-1)q}\dfrac{1}{z^{2(q-1)q}}
  +\cdots \notag\\
  -\left(\dfrac{d}{c}\right)&^{(q-1)^2}\dfrac{1}{z^{(q-1)^2}}-\left(\dfrac{d}{c}\right)^{(2q-1)(q-1)}
  \dfrac{1}{z^{(2q-1)(q-1)}}
  -\left(\dfrac{d}{c}\right)^{(3q-1)(q-1)}\dfrac{1}{z^{(3q-1)(q-1)}}-\cdots\bigg). \label{e:mdd5}
\end{align}
Hence we get the Laurent series expansion of $E_{q-1}(z)$:
\[
E_{q-1}(z)=S_1 + S_2
\]
which is the summation of (\ref{e:mdd4}) and (\ref{e:mdd5}). In the
equation (\ref{e:3.9}) for $\Delta(z)$, the part with $E_{q^2-1}(z)$
doesn't contribute to the residue of $z^i\Delta(z)dz$ on the region
$D=\{z\in \Omega:\,\,|\pi|<|z|<1\}$ for $0\le i\le q^2-3$, since
\[
E_{q^2-1}(z)=\sum\limits_{(0,0)\ne
(c,\,d)\in{\A}^2}\dfrac{1}{(cz+d)^{q^2-1}}\,,
\]
and the residue of $z^i dz/(cz+d)^{q^2-1}$ is $0$ as long as $0\le
i\le q^2-3$. Therefore we need only compute the coefficient of
$dz/z$ in $(T^q-T)^q E_{q-1}(z)^{q+1}z^i dz$. So consider
\begin{equation}\label{e:mdd6}
E_{q-1}(z)^{q+1}=(S_1+S_2)^{q+1}=S_1^{q+1}+S_1^qS_2+S_2^qS_1+S_2^{q+1}.
\end{equation}
It is clear that the terms $S_1^{q+1}$ and $S_2^{q+1}$ in (\ref{e:mdd6}) don't
contribute to the wanted residues. Hence we need only compute the expansions
of $S_1^q S_2$ and $S_2^q S_1$. As
\begin{align}
S_1^q S_2=\left(\sum\limits_{\substack{(c,\,d)\in{\A}^2\\\deg(d)\ge\deg(c)}}
\dfrac{1}{d^{(q-1)q}}\left(\sum\limits_{i=0}^\infty \left(\dfrac{c}{d}\right)^{i(q-1)q^2} z^{i(q-1)q^2}
      -\sum\limits_{j=1}^\infty \left(\dfrac{c}{d}\right)^{(jq-1)(q-1)q} z^{(jq-1)(q-1)q}\right)\right)\cdot\notag\\
\left(\sum\limits_{\substack{(c,\,d)\in{\A}^2\\\deg(c)\ge\deg(d)+1}}\dfrac{1}{c^{q-1}z^{q-1}}
\left(\sum\limits_{s=0}^\infty \left(\dfrac{d}{c}\right)^{s(q-1)q}\dfrac{1}{z^{s(q-1)q}}
-\sum\limits_{t=1}^\infty \left(\dfrac{d}{c}\right)^{(tq-1)(q-1)}\dfrac{1}{z^{(tq-1)(q-1)}}\right)\right),\label{e:mdd7}
\end{align}
we consider all possible powers $(1/z)^m$ for $m\in{\Z}$ in the above expansion:
\begin{itemize}
\item[(1).] for $1\le m=-i(q-1)q^2+s(q-1)q+q-1\le q^2-2$ with $i\ge 0$ and $s\ge 0$, then the only solution for $m$ is
    $m=q-1$ when $s=iq$ for $i\ge 0$; in general, $m=(q-1)+l(q-1)q$ for $l\in {\Z}$;
\item[(2).] for $1\le m=-i(q-1)q^2+(tq-1)(q-1)+q-1\le q^2-2$ with $i\ge 0$ and $t\ge 1$, then the only solution for $m$
    is $m=(q-1)q$ when $t=iq+1$ for $i\ge 0$; in general, $m=l(q-1)q$ for $l\in{\Z}$;
\item[(3).] for $1\le m=-(jq-1)(q-1)q+s(q-1)q+q-1\le q^2-2$ with $j\ge 1$ and $s\ge 0$, then the only solution for $m$ is
    $m=q-1$ when $s=jq-1$ for $j\ge 1$; in general, $m=(q-1)+l(q-1)q$ for $l\in{\Z}$;
\item[(4).] for $1\le m=-(jq-1)(q-1)q+(tq-1)(q-1)+q-1\le q^2-2$ with $j\ge 1$ and $t\ge 1$, then the only solution for $m$
    is $m=(q-1)q$ when $t=jq$ for $j\ge 1$; in general, $m=l(q-1)q$ for $l\in{\Z}$.
\end{itemize}
Therefore we get from (\ref{e:mdd7})
\begin{equation}\label{e:mdd8}
S_1^qS_2=\dfrac{\Upsilon_{10}}{z^{q-1}} + \dfrac{\Xi_{11}}{z^{(q-1)q}}+
\sum\limits_{0\ne i\in{\Z}}\dfrac{\Upsilon_{1i}}{z^{q-1+i(q-1)q}} + \sum\limits_{1\ne j\in{\Z}}\dfrac{\Xi_{1j}}{z^{j(q-1)q}}.
\end{equation}
where
\begin{align*}
\Upsilon_{10}=&\sum\limits_{i=0}^\infty\Bigg(\sum\limits_{\substack{(c,\,d)\in {\A}^2\\\deg(d)\ge\deg(c)}}
  \dfrac{1}{d^{(q-1)q}}\left(\dfrac{c}{d}\right)^{i(q-1)q^2}\Bigg)\cdot
  \Bigg(\sum\limits_{\substack{(c,\,d)\in{\A}^2\\\deg(c)\ge\deg(d)+1}}
  \dfrac{1}{c^{q-1}}\left(\dfrac{d}{c}\right)^{i(q-1)q^2}\Bigg)\\
&-\sum\limits_{j=1}^\infty\Bigg(\sum\limits_{\substack{(c,\,d)\in{\A}^2\\\deg(d)\ge\deg(c)}}
  \dfrac{1}{d^{(q-1)q}}\left(\dfrac{c}{d}\right)^{(jq-1)(q-1)q}\Bigg)\cdot
  \Bigg(\sum\limits_{\substack{(c,\,d)\in{\A}^2\\\deg(c)\ge\deg(d)+1}}
  \dfrac{1}{c^{q-1}}\left(\dfrac{d}{c}\right)^{(jq-1)(q-1)q}\Bigg),\\
\Xi_{11}=&-\sum\limits_{i=0}^\infty\Bigg(\sum\limits_{\substack{(c,\,d)\in {\A}^2\\\deg(d)\ge\deg(c)}}
  \dfrac{1}{d^{(q-1)q}}\left(\dfrac{c}{d}\right)^{i(q-1)q^2}\Bigg)\cdot
  \Bigg(\sum\limits_{\substack{(c,\,d)\in{\A}^2\\\deg(c)\ge\deg(d)+1}}
  \dfrac{1}{c^{q-1}}\left(\dfrac{d}{c}\right)^{((iq+1)q-1)(q-1)}\Bigg)\\
&+\sum\limits_{j=1}^\infty\Bigg(\sum\limits_{\substack{(c,\,d)\in{\A}^2\\\deg(d)\ge\deg(c)}}
  \dfrac{1}{d^{(q-1)q}}\left(\dfrac{c}{d}\right)^{(jq-1)(q-1)q}\Bigg)\cdot
  \Bigg(\sum\limits_{\substack{(c,\,d)\in{\A}^2\\\deg(c)\ge\deg(d)+1}}
  \dfrac{1}{c^{q-1}}\left(\dfrac{d}{c}\right)^{(jq^2-1)(q-1)}\Bigg).
\end{align*}
We compute the expansion of $S_2^q S_1$ in the same way. In the following expansion of
\begin{align}
&S_2^q S_1\notag \\
=&\sum\limits_{\substack{(c,\,d)\in{\A}^2\\\deg(c)\ge\deg(d)+1}}\dfrac{1}{c^{(q-1)q}z^{(q-1)q}}
\left(\sum\limits_{s=0}^\infty \left(\dfrac{d}{c}\right)^{s(q-1)q^2}\dfrac{1}{z^{s(q-1)q^2}}
-\sum\limits_{t=1}^\infty \left(\dfrac{d}{c}\right)^{(tq-1)(q-1)q}\dfrac{1}{z^{(tq-1)(q-1)q}}\right)\notag\\
&\cdot\sum\limits_{\substack{(c,\,d)\in{\A}^2\\\deg(d)\ge\deg(c)}}
\dfrac{1}{d^{(q-1)}}\left(\sum\limits_{i=0}^\infty \left(\dfrac{c}{d}\right)^{i(q-1)q} z^{i(q-1)q}
      -\sum\limits_{j=1}^\infty \left(\dfrac{c}{d}\right)^{(jq-1)(q-1)} z^{(jq-1)(q-1)}\right),\label{e:mdd9}
\end{align}
we find all possible powers $(1/z)^m$ for $m\in{\Z}$:
\begin{itemize}
\item[(1).] for $1\le m=s(q-1)q^2+(q-1)q-i(q-1)q\le q^2-2$ with $i\ge 0$ and $s\ge 0$, the only solution is $m=(q-1)q$ when $i=sq$ for $s\ge 0$; in general, $m=l(q-1)q$ for $l\in{\Z}$;
\item[(2).] for $1\le m=s(q-1)q^2+(q-1)q-(jq-1)(q-1)\le q^2-2$ with $j\ge 1$ and $s\ge 0$, the only solution is $m=q-1$ when $j=sq-1$ for $s\ge 1$; in general, $m=(q-1)+l(q-1)q$ for $l\in{\Z}$;
\item[(3).] for $1\le m=(tq-1)(q-1)q+(q-1)q-i(q-1)q\le q^2-2$ with $i\ge 0$ and $t\ge 1$, the only solution is $m=(q-1)q$ when $i=tq-1$ for $t\ge 1$; in general, $m=l(q-1)q$ for $l\in{\Z}$;
\item[(4).] for $1\le m=(tq-1)(q-1)q+(q-1)q-(jq-1)(q-1)\le q^2-2$ with $j\ge 1$ and $t\ge 1$, the only solution is $m=q-1$ when $j=tq$ for $t\ge 1$; in general, $m=(q-1)+l(q-1)q$ for $l\in{\Z}$.
\end{itemize}
Therefore we get from (\ref{e:mdd9})
\begin{equation}\label{e:mdd10}
S_2^q S_1=\dfrac{\Upsilon_{20}}{z^{q-1}} + \dfrac{\Xi_{21}}{z^{(q-1)q}}+
\sum\limits_{0\ne i\in{\Z}}\dfrac{\Upsilon_{2i}}{z^{q-1+i(q-1)q}} + \sum\limits_{1\ne j\in{\Z}}\dfrac{\Xi_{2j}}{z^{j(q-1)q}}.
\end{equation}
where
\begin{align*}
\Upsilon_{20}=&-\sum\limits_{s=1}^\infty\Bigg(\sum\limits_{\substack{(c,\,d)\in{\A}^2\\\deg(c)\ge\deg(d)+1}}
  \dfrac{1}{c^{(q-1)q}}\left(\dfrac{d}{c}\right)^{s(q-1)q^2}\Bigg)\cdot
  \Bigg(\sum\limits_{\substack{(c,\,d)\in{\A}^2\\\deg(d)\ge\deg(c)}}
  \dfrac{1}{d^{q-1}}\left(\dfrac{c}{d}\right)^{((sq-1)q-1)(q-1)}\Bigg)\\
&+\sum\limits_{t=1}^\infty\Bigg(\sum\limits_{\substack{(c,\,d)\in{\A}^2\\\deg(c)\ge\deg(d)+1}}
  \dfrac{1}{c^{(q-1)q}}\left(\dfrac{d}{c}\right)^{(tq-1)(q-1)q}\Bigg)\cdot
  \Bigg(\sum\limits_{\substack{(c,\,d)\in{\A}^2\\\deg(d)\ge\deg(c)}}
  \dfrac{1}{d^{q-1}}\left(\dfrac{c}{d}\right)^{(tq^2-1)(q-1)}\Bigg),\\
\Xi_{21}=&\sum\limits_{s=0}^\infty\Bigg(\sum\limits_{\substack{(c,\,d)\in{\A}^2\\\deg(c)\ge\deg(d)+1}}
  \dfrac{1}{c^{(q-1)q}}\left(\dfrac{d}{c}\right)^{s(q-1)q^2}\Bigg)\cdot
  \Bigg(\sum\limits_{\substack{(c,\,d)\in{\A}^2\\\deg(d)\ge\deg(c)}}
  \dfrac{1}{d^{q-1}}\left(\dfrac{c}{d}\right)^{s(q-1)q^2}\Bigg)\\
&-\sum\limits_{t=1}^\infty\Bigg(\sum\limits_{\substack{(c,\,d)\in{\A}^2\\\deg(c)\ge\deg(d)+1}}
  \dfrac{1}{c^{(q-1)q}}\left(\dfrac{d}{c}\right)^{(tq-1)(q-1)q}\Bigg)\cdot
  \Bigg(\sum\limits_{\substack{(c,\,d)\in{\A}^2\\\deg(d)\ge\deg(c)}}
  \dfrac{1}{d^{q-1}}\left(\dfrac{c}{d}\right)^{(tq-1)(q-1)q}\Bigg).
\end{align*}
Hence we have the expansion on $D=\{z\in\Omega:\,\,|\pi|<|z|<1\}$ from (\ref{e:mdd8}) and (\ref{e:mdd10})
\begin{align}
&\Delta(z)=\dfrac{\Upsilon_0}{z^{q-1}}+\dfrac{\Xi_1}{z^{(q-1)q}}+
\sum\limits_{m\ge 0}\ast\cdot z^m + \sum\limits_{m\le -(q^2-2)}\ast\cdot z^m,\label{e:mdd11} \\
&\Upsilon_0=(T^q-T)^q\left(\Upsilon_{10}+\Upsilon_{20}\right),\label{e:Upsilon} \\
&\Xi_1=(T^q-T)^q\left(\Xi_{11}+\Xi_{21}\right).\label{e:Xi}
\end{align}
%\begin{equation}\label{e:mdd11}
%\begin{split}
%&\Delta(z)=\dfrac{\Upsilon_0}{z^{q-1}}+\dfrac{\Xi_1}{z^{(q-1)q}}+
%\sum\limits_{m\ge 0}\ast\cdot z^m + \sum\limits_{m\le -(q^2-1)}\ast\cdot z^m, \\
%&\Upsilon_0=(T^q-T)^q\left(\Upsilon_{10}+\Upsilon_{20}\right), \\
%&\Xi_1=(T^q-T)^q\left(\Xi_{11}+\Xi_{21}\right).
%\end{split}
%\end{equation}
Let $\mu_{\Delta}$ be the measure associated to the Drinfeld discriminant $\Delta(z)$, which is
by definition, the measure associated to the harmonic cocycle ${\mathfrak c}_{\Delta}$ given in equation (\ref{e:fbtt4}).
Then for $0\le i\le q^2-3$,
\begin{align*}
\int_{U(e_0)} x^id\mu_{\Delta}(x)=\,&{\mathfrak c}_{\Delta}(e)(X^iY^{q^2-3-i})={\rm Res}_{e_0}z^i\Delta(z)dz\\
=\,&{\rm Res}_{e_0}\left(\Upsilon_0\dfrac{z^idz}{z^{q-1}}+\Xi_1\dfrac{z^idz}{z^{(q-1)q}}\right)
\end{align*}
where $U(e_0)=U(\Lambda_0\Lambda_1)=\{x\in {\bk}_\infty: \,\, |x|\ge |\pi|^{-1}\}\cup\{\infty\}
=B_{\infty}(|\pi|^{1})$ is given in Example \ref{example:fbtt1}. From the above discussion, we get
\begin{lemma}\label{lem:mdd2}
For $0\le i\le q^2-3$, we have the integral values of $x^i$ over $U(e_0)$ against $\mu_{\Delta}$:
\begin{align*}
&{\rm (1) \,}\int_{U(e_0)}x^id\mu_{\Delta}(x)=0 \quad\text{ if $i\ne q-2, \,q^2-q-1$};
\qquad\qquad\qquad\qquad\qquad\qquad\qquad\\
&{\rm (2) \,}\int_{U(e_0)}x^{q-2}d\mu_{\Delta}(x)=\Upsilon_0;\\
&{\rm (3) \,}\int_{U(e_0)}x^{q^2-q-1}d\mu_{\Delta}(x)=\Xi_1.
\end{align*}
\end{lemma}

\begin{lemma}\label{lem:mdd3}
$\Xi_1=0$.
\end{lemma}
\begin{proof} We consider the edges around the vertex ${\Lambda}_0$ of the Bruhat-Tits tree $\mathcal T$:
\begin{equation}\label{figure}
\xymatrix@=48pt{
&                            &{\circ}\ar@{}[d]\ar@{}^{[\pi{\bv}_1,\, b{\bv}_1+{\bv}_2]}                     &\\
&{\circ}\ar@{->}[r]^{e_{-1}}\ar@{ }_{\Lambda_{-1}}\ar@{<.}[l] &{\circ}\ar@{<-}[u]^{e_{b,0}}\ar@{->}[r]^{e_0}
\ar@{}_{\Lambda_0} &{\circ}\ar@{.>}[r]
\ar@{}_{\Lambda_1} & \\
&                            &\ar@{.>}[u]  &
}
\end{equation}
where the notations are as in Section \ref{btt}, and $b\in {\Fq}^*$. Under the correspondence (\ref{e:2.4}), the vertex
$[\pi{\bv}_1, \, b{\bv}_1+{\bv}_2]$ corresponds to the matrix
$\left(\begin{array}{cc}\pi & b\\0 & 1\end{array}\right)$ and $\Lambda_{-1}$
corresponds to the matrix $\left(\begin{array}{cc}\pi&0\\0&1\end{array}\right)$.
Let $\gamma_b=\left(\begin{array}{cc}1&b\\0&1\end{array}\right)$ for $b\in{\Fq}^*$,
and $\delta=\left(\begin{array}{cc}0&1\\1&0\end{array}\right)$. Then
$(\gamma_b^{-1})^T=\left(\begin{array}{cc}1&0\\-b&1\end{array}\right)$. The matrices
$\gamma_b$ and $\delta$ fix the vertex $\Lambda_0$ under the ``$\ast$" action.
And by performing column operations, we see that
\begin{alignat*}{2}
\delta\ast\Lambda_1=\left(\begin{array}{cc}0&1\\1&0\end{array}\right)
  \left(\begin{array}{cc}\pi^{-1}&0\\0&1\end{array}\right)=
  \left(\begin{array}{cc}0&1\\\pi^{-1}&1\end{array}\right)
  &&&\sim\left(\begin{array}{cc}\pi&0\\0&1\end{array}\right)
  =\Lambda_{-1},\\
\gamma_b\ast\Lambda_1=\left(\begin{array}{cc}1&0\\-b&1\end{array}\right)
  \left(\begin{array}{cc}\pi^{-1}&0\\0&1\end{array}\right)=
  \left(\begin{array}{cc}\pi^{-1}&0\\-b\pi^{-1}&1\end{array}\right)
  &&&\sim\left(\begin{array}{cc}\pi&-b^{-1}\\0&1\end{array}\right)\\
  &&&\quad=[\pi{\bv}_1,\,-b^{-1}{\bv}_1+{\bv}_2],
\end{alignat*}
where the notion $\sim$ denotes that two matrices are in the same class in the
quotient space ${\GL2}(\bk_\infty)/({\bk^*_\infty}\cdot {\GL2}(\A_\infty))$.
Hence $\overline{e_{b,0}}=\gamma_{-b^{-1}}\ast e_0$ and
$\overline{e_{-1}}=\delta \ast e_0$,
where $\overline{e}$ denotes the edge with opposite orientation of $e$.

By the definitions in Section \ref{fbtt}, we have for an edge $e=\gamma\ast {e_0}$ with
$\gamma=\left(\begin{array}{cc}a&b\\c&d\end{array}\right)\in {\GL2}({\bk}_{\infty})^+$ and
$0\le i\le q^2-3$,
\begin{align*}
\int_{U(e)}x^id\mu_{\Delta}(x)&={\mathfrak c}_{\Delta}(e)(X^iY^{q^2-3-i})\\
&={\rm Res}_e z^i\Delta(z)dz\\
&={\rm Res}_{e_0} (\gamma z)^i\Delta(\gamma z) d\gamma z\\
&=\int_{U(e_0)}{\det(\gamma)}(ax+b)^i(cx+d)^{q^2-3-i}d\mu_{\Delta}(x).
\end{align*}
Therefore for $0\le j\le q^2-3$,
\begin{align*}
\int_{U(\gamma_b\ast e_0)} x^j d\mu_{\Delta}(x)&=\int_{U(e_0)} (x+b)^jd\mu_{\Delta}(x),\\
\int_{U(e_{-1})} x^jd\mu_{\Delta}(x)&=-\int_{U(\delta\ast e_0)}x^jd\mu_{\Delta}(x)=\int_{U(e_0)}x^{q^2-3-j}d\mu_{\Delta}(x).
\end{align*}
Hence from
\begin{align*}
&\sum\limits_{b\in{\Fq}^{*}}\int_{U(\gamma_b\ast e_0)}x^jd\mu_{\Delta}(x)+\int_{U(e_0)}x^jd\mu_{\Delta}(x)
-\int_{U(e_{-1})}x^jd\mu_{\Delta}(x)\\
=&\sum\limits_{b\in{\Fq}^{*}}\int_{U(\gamma_{-b^{-1}}\ast e_0)}x^jd\mu_{\Delta}(x)+\int_{U(e_0)}x^jd\mu_{\Delta}(x)
-\int_{U(e_{-1})}x^jd\mu_{\Delta}(x)\\
=&-\left(\sum\limits_{b\in{\Fq}^{*}}\int_{U(e_{b,0})}x^jd\mu_{\Delta}(x)+\int_{U(\overline{e_0})}x^jd\mu_{\Delta}(x)
+\int_{U(e_{-1})}x^jd\mu_{\Delta}(x)\right)\\
=&\,0,
\end{align*}
we get
\begin{equation*}
\int_{U(e_0)}x^{q^2-3-j}d\mu_{\Delta}(x)
=\sum\limits_{b\in{\Fq}^*}\int_{U(e_0)}(x+b)^jd\mu_{\Delta}(x)+\int_{U(e_0)}x^jd\mu_{\Delta}(x).
\end{equation*}
We plug in $j=q-2$ in the above equation and apply (1) of Lemma \ref{lem:mdd2}, then
get
\begin{equation*}
\int_{U(e_0)}x^{q^2-q-1}d\mu_{\Delta}(x)=0,
\end{equation*}
therefore $\Xi_1=0$.
\end{proof}

Recall that the $\Lambda_n$ ($n\ge 0$) are the vertices of the half line (\ref{e:2.5}), and $e_n=\Lambda_n\Lambda_{n+1}$
is the oriented edge with origin $\Lambda_n$ and terminal $\Lambda_{n+1}$.
\begin{theorem}\label{thm:mdd2}
The measure $\mu_{\Delta}$ on ${\mathbb P}^1({\bk}_\infty)$ associated
to the Drinfeld discriminant $\Delta(z)$ is
$h$-admissible with $h$ being the smallest integer greater than or
equal to $(q^2-3)/2$, and is completely determined by the following
values:
\begin{align*}
&{\rm (1)\,} \int_{U(e_0)} x^jd\mu_{\Delta}(x)=\begin{cases}0,\, &\text{ if \, $0\le j\le q^2-3$, $j\ne q-2$},\\
\Upsilon_0,\, &\text{ if  $j=q-2$}.\end{cases}\qquad\qquad\qquad\\
&{\rm (2)\,} \int_{U(e_n)}x^jd\mu_{\Delta}(x)=0,\, \text{ for
\,$0\le j\le q^2-3$ and any $n\ge 1$.}
\end{align*}
\end{theorem}
\begin{proof}
That the measure $\mu_{\Delta}$ is $h$-admissible follows from Corollary
\ref{cor:fbtt1} directly.

Since the fundamental domain of $\Gamma\backslash{\mathcal T}$ is the half line (\ref{e:2.1})
consisting of $\Lambda_n$ ($n\ge 0$) (for both the ordinary action
and the action ``$\ast$") and the measure $\mu_{\Delta}$ satisfies the
equation (\ref{e:fbtt5}), $\mu_{\Delta}$ is determined by the values $\int_{U(e_n)}x^jd\mu_{\Delta}(x)$
for $0\le j\le q^2-3$ and all integers $n\ge 0$.

From Proposition \ref{prop:2.1}, we know that under the ordinary
action, the group
\begin{equation*}
\Gamma_n=\left\{\left(\begin{array}{cc}a&b\\0&d\end{array}\right):\,\,
a,d\in {\Fq}^*, \,b\in {\Fq}[T], \,\deg_T(b)\le n\right\}, \quad
n\ge 1
\end{equation*}
fixes the edge $\Lambda_n\Lambda_{n+1}$, and acts transitively on
the set of edges with origin $\Lambda_n$ but distinct from the edge
$\Lambda_n\Lambda_{n+1}$. In fact, the element $\gamma_b=\left(\begin{array}{cc}1&b\\0&1\end{array}\right)$
for $b=\beta T^n$ with $\beta\in{\Fq}$ fixes $\Lambda_n$ and $\Lambda_{n+1}$ and takes
$\Lambda_{n-1}$ to the vertex $[\pi^{-(n-1)}{\bv}_1,\,\beta\pi^{-n}{\bv}_1+{\bv}_2]$, since $\Lambda_{n-1}$
is represented by the lattice $L_{n-1}={\A}_\infty\oplus \pi^{n-1}{\A}_\infty\sim \pi^{-(n-1)}{\A}_\infty
\oplus{\A}_\infty$, which is represented by the matrix $\left(\begin{array}{cc}1&0\\0&\pi^{n-1}\end{array}\right)$,
and
\begin{equation*}
\left(\begin{array}{cc}1&\beta T^{n}\\0&1\end{array}\right)
\left(\begin{array}{cc}1&0\\0&\pi^{n-1}\end{array}\right)=
\left(\begin{array}{cc}1&\beta T\\0&\pi^{n-1}\end{array}\right)\sim
\left(\begin{array}{cc}\pi^{-(n-1)}&\beta\pi^{-n}\\0&1\end{array}\right)
\end{equation*}
where the notion `$\sim$' means that the two matrices are the same in
${\GL2}(\bk_\infty)/({\bk^*_\infty}\cdot {\GL2}(\A_\infty))$ as before, thus they
represent the same vertex of the Bruhat-Tits tree $\mathcal T$ in terms of
the correspondence (\ref{e:2.2}).

Conclusion (1) of this theorem comes from Lemma \ref{lem:mdd2} and Lemma \ref{lem:mdd3} directly.

For conclusion (2), we consider the case $n=1$ at first. The $q+1$ edges with terminal vertex $\Lambda_1$
are $e_{b,1}$ (for $b=\beta T$ and $\beta\in{\Fq}$) and $\overline{e_1}$, as illustrated in the following
diagram:
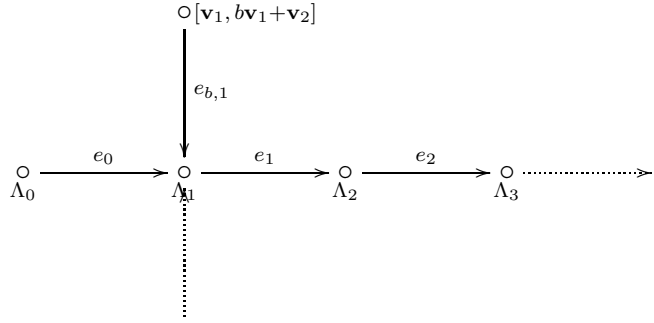
\begin{equation*}
\xymatrix@=48pt{
&                                           &{\circ}\ar@{->}[d]^{e_{b,1}}\ar@{}^{[{\bv}_1,\, b{\bv}_1+{\bv}_2]} & & &\\
&{\circ}\ar@{->}[r]^{e_0}\ar@{}_{\Lambda_0} &{\circ}\ar@{->}[r]^{e_1}\ar@{}_{\Lambda_1}
    &{\circ}\ar@{->}[r]^{e_2}\ar@{}_{\Lambda_2} &{\circ}\ar@{.>}[r]\ar@{}_{\Lambda_3}&\\
&                                           &\ar@{.>}[u] & & &
}
\end{equation*}
where the origin vertex of $e_{b,1}$ is
\begin{equation*}
(\gamma_b^{-1})^T\ast \left(\begin{array}{cc}1&0\\0&1\end{array}\right)
=\left(\begin{array}{cc}1&\beta T\\0&1\end{array}\right)\cdot
\left(\begin{array}{cc}1&0\\0&1\end{array}\right)
\sim\left(\begin{array}{cc}
1&\beta\pi^{-1}\\0&1\end{array}\right)\sim\left(\begin{array}{cc}1&b\\0&1\end{array}
\right).
\end{equation*}
Therefore
\begin{align}
&\int_{U(e_1)}x^jd\mu_{\Delta}(x)\notag\\
=&\sum\limits_{b\in T\cdot{\Fq}^*}\int_{U(e_{b,1})}x^jd\mu_{\Delta}(x)+\int_{U(e_0)}x^jd\mu_{\Delta}(x)\notag\\
=&\sum\limits_{b\in T\cdot{\Fq}^*}\int_{U((\gamma_b^{-1})^T\ast e_0)}x^jd\mu_{\Delta}(x)
   + \int_{U(e_0)}x^jd\mu_{\Delta}(x)\notag\\
=&\sum\limits_{b\in T\cdot{\Fq}^*}\int_{U(e_0)}x^j(-bx+1)^{q^2-3-j}d\mu_{\Delta}(x)
   + \int_{U(e_0)}x^jd\mu_{\Delta}(x)\notag\\
=&\sum\limits_{i=0}^{q^2-3-j}\binom{q^2-3-j}{i}\sum\limits_{b\in T\cdot{\Fq}^*}
  (-b)^i\int_{U(e_0)}x^{i+j}d\mu_{\Delta}(x)+ \int_{U(e_0)}x^jd\mu_{\Delta}(x)
  \label{e:mdd12}.
\end{align}
In the expression (\ref{e:mdd12}), the integral of the first term is equal to $0$
unless $i+j=q-2$, thus $\int_{U(e_1)}x^jd\mu_{\Delta}(x)=0$ if $j>q-2$ by Lemma
\ref{lem:mdd2} and \ref{lem:mdd3}.
For $j\le q-2$, we get from (\ref{e:mdd12}) that
\begin{equation*}
\int_{U(e_1)}x^jd\mu_{\Delta}(x)
=\binom{q^2-3-j}{q-2-j}\sum\limits_{b\in T\cdot{\Fq}^*}(-b)^{q-2-j}\int_{U(e_0)}
x^{q-2}d\mu_{\Delta}(x)+\int_{U(e_0)}x^jd\mu_{\Delta}(x).
\end{equation*}
The summation $\sum_{b\in T\cdot{\Fq}^*}(-b)^{q-2-j}=0$ unless $j=q-2$, so it is
clear that the integral $\int_{U(e_1)}x^jd\mu_{\Delta}(x)$ is equal to $0$ for $j<q-2$,
because the integral $\int_{U(e_0)}x^jd\mu_{\Delta}(x)$ is also equal to $0$ in
this case. And for $j=q-2$, we have
\begin{align*}
\int_{U(e_1)}x^{q-2}d\mu_{\Delta}(x)&=\binom{q^2-q-1}{0} \sum\limits_{b\in T\cdot{\Fq}^*} 1
\cdot\int_{U(e_0)}
x^{q-2}d\mu_{\Delta}(x)+\int_{U(e_0)}x^{q-2}d\mu_{\Delta}(x)\\
&=0.
\end{align*}
Therefore we have showed that
\begin{equation*}
\int_{U(e_1)}x^{j}d\mu_{\Delta}(x)=0
\end{equation*}
for $0\le j\le q^2-3$.

For $n>1$, we can show in a similar way that $\int_{U(e_n)}x^{j}d\mu_{\Delta}(x)=0$ for
$0\le j\le q^2-3$ by induction on $n$.
\end{proof}

\begin{cor}\label{cor:mdd1}
$\Upsilon_0\ne 0$.
\end{cor}
\begin{proof}
Suppose $\Upsilon_0=0$. Then $\int_{U(e_0)}x^jd\mu_{\Delta}(x)=0$ for all integers $j$
between $0$ and $q^2-3$, thus we can show in the same way as the proof of Theorem
\ref{thm:mdd2} that $\int_{U(e_n)}x^jd\mu_{\Delta}(x)=0$ for all integers $n\ge 0$
and $0\le j\le q^2-3$, which implies that $\int_{U(e)}x^jd\mu_{\Delta}(x)=0$
for all edges $e$ of the Bruhat-Tits tree $\mathcal T$. Therefore the harmonic cocycle ${\mathfrak c}_{\Delta}=0$, and then $\Delta(z)=0$ by Theorem \ref{thm:fbtt3}, which is absurd.
\end{proof}

\begin{cor}\label{cor:mdd2}
For any edge $e$ of the Bruhat-Tits tree $\mathcal T$ and $0\le j\le q^2-3$, we have
\begin{equation*}
\int_{U(e)}x^j d\mu_{\Delta}(x)=r(e,j)\cdot \Upsilon_0
\end{equation*}
where $r(e,j)\in{\A}$.
\end{cor}
\begin{proof}
The edge $e$ can be obtained as $e=\gamma\ast e_n$ for some integer $n\ge 0$ and some
$\gamma=\left(\begin{array}{cc}a&b\\c&d\end{array}\right)\in {\Gamma}$. Hence
\begin{align}
\int_{U(e)} x^jd\mu_{\Delta}(x)&=\int_{U(\gamma\ast e_n)}x^jd\mu_{\Delta}(x)\notag\\
&=\int_{U(e_n)}\det(\gamma) (ax+b)^j (cx+d)^{q^2-3-j}d\mu_{\Delta}(x).\label{e:mdd13}
\end{align}
In the integral of (\ref{e:mdd13}), the integrand is a polynomial with coefficients
in $\A$, therefore the integral $\int_{U(e)}x^jd\mu_{\Delta}(x)$ is $0$ if $n\ne 0$.
And if $n=0$, the integral $\int_{U(e)}x^jd\mu_{\Delta}(x)$ is equal to $r(e,j)\cdot
\int_{U(e_0)}x^{q-2}d\mu_{\Delta}(x)$ $=r(e,j)\cdot\Upsilon_0$ for some $r(e,j)\in {\A}$.
\end{proof}

%=================================================================================================================

\section{Complements}\label{comp}
In Theorem \ref{thm:mdd2}, we have determined the integrals
$\int_{U(e_0)}x^jd\mu_{\Delta}(x)$ for $0\le j\le q^2-3$ and we can
see what the values $\int_{U(e)}x^jd\mu_{\Delta}(x)$ look like in
Corollary \ref{cor:mdd2} for a general edge $e$ of $\mathcal T$ and
$0\le j\le q^2-3$. In general, we set
\begin{equation}\label{e:comp1}
L(\Delta;e;j)=\int_{U(e)}x^{j-1}d\mu_{\Delta}(x) \quad\text{ for
$e\in E_{\mathcal T}$ and $j\in{\Z}$}
\end{equation}
as given in the equation (\ref{e:fbtt22}). Since $\delta=\left(\begin{array}{cc}
0&1\\1&0\end{array}\right)\in \Gamma$, we have the following equation for these values:
\begin{equation}\label{e:comp2}
L(\Delta;\delta\ast e;j)=-L(\Delta;e,q^2-1-j) \quad\text{ for $e\in E_{\mathcal T}$ and
$j\in{\Z}$}.
\end{equation}
When the edge $e$ is $\overline{e_{1,0}}$ in (\ref{figure}), the
corresponding $U(\overline{e_{1,0}})$ is the set ${\mathcal U}_1$ of
$1$-units of ${\bk}_\infty$, so $L(\Delta;\overline{e_{1,0}};j)$ can
be extended to $L(\Delta;\overline{e_{1,0}};s)$ for $s\in{\Z}_p$,
which is $L_{\Delta}(s)$ in terms of the notations in
(\ref{e:lfunction}), and (\ref{e:comp2}) becomes the functional
equation
\begin{equation}\label{e:compfeq}
L_{\Delta}(s)=-L_{\Delta}(q^2-1-s)
\end{equation}
as pointed out by D. Goss in \cite{Go1992}. About the measure associated to
the Drinfeld discriminant $\Delta$, it is also interesting to see what the values
$L_{\Delta}(j)$ might be for $j\ge 1$.

As a special case of integrals of functions in $C^h({\P1}(\bk_\infty))$ against the measure $\mu_{\Delta}$ (see Corollary \ref{cor:fbtt1} and Theorem \ref{thm:fbtt3}),
the values $L(\Delta;e;j)$ for $j\ge q^2-1$ or $j\le 0$ are obtained through a
limit process by those for $1\le j\le q^2-2$, but they can also be calculated by applying the equation (\ref{e:fbtt21}) in terms of the residues. At first, we will
consider $L(\Delta;e_0;j+1)$ $=\int_{U(e_0)}x^jd\mu_{\Delta}(x)$ by expanding the
Drinfeld discriminant $\Delta(z)$ over the region
$\lambda^{-1}(e_0)=\{z\in \Omega:\,\,|\pi|<|z|<1\}$ better than
the equation (\ref{e:mdd11}), where the terms with $z^{m}$ for $m\ge 0$ or
$m\le -(q^2-2)$ are omitted.

In the equation (\ref{e:mdd6}), we've already calculated $S_1^qS_2$ in (\ref{e:mdd8})
and $S_2^qS_1$ in (\ref{e:mdd10}). And we do similar calculations to get
\begin{align}
&S_1^{q+1}=\sum\limits_{l=0}^\infty \ast\cdot z^{l(q-1)q}
  +\sum\limits_{l=1}^\infty \ast\cdot z^{-(q-1)+l(q-1)q},\label{e:comp3}\\
&S_2^{q+1}=\sum\limits_{l=1}^\infty\dfrac{\ast}{z^{q-1+l(q-1)q}}
  +\sum\limits_{l=2}^\infty\dfrac{\ast}{z^{l(q-1)q}}.\label{e:comp4}
\end{align}
Therefore we get from (\ref{e:mdd8}), (\ref{e:mdd10}), (\ref{e:comp3}), and
(\ref{e:comp4}) that
\begin{equation*}
(T^q-T)^qE_{q-1}(z)^{q+1}=\dfrac{\Upsilon_0}{z^{q-1}}+\dfrac{\Xi_1}{z^{(q-1)q}}+
\sum\limits_{0\ne l\in{\Z}}\dfrac{\ast}{z^{q-1+l(q-1)q}}
+ \sum\limits_{1\ne l\in{\Z}}\dfrac{\ast}{z^{l(q-1)q}}.
\end{equation*}
We also expand $E_{q^2-1}(z)$ over the region $\lambda^{-1}(e_0)$ as:
\begin{equation*}
E_{q^2-1}(z)=\sum\limits_{l\in{\Z}}\dfrac{\ast}{z^{l(q-1)q^2}}
+\sum\limits_{l\in{\Z}}\dfrac{\ast}{z^{(q-1)+(q-1)q+l(q-1)q^2}}.
\end{equation*}
Therefore we get the expansion of $\Delta(z)$ over $\lambda^{-1}(e_0)$:
\begin{align}
\Delta(z)&=(T^{q^2}-T)E_{q^2-1}(z)+(T^q-T)^qE_{q-1}(z)^{q+1}\notag\\
&=\sum\limits_{i\in{\Z}}\dfrac{\Upsilon_i}{z^{q-1+i(q-1)q}}+
\sum\limits_{i\in{\Z}}\dfrac{\Xi_i}{z^{i(q-1)q}}\label{e:comp5}
\end{align}
where $\Upsilon_0$ and $\Xi_1$ are given in (\ref{e:Upsilon}) and (\ref{e:Xi}),
respectively.
\begin{prop}\label{prop:comp1} \begin{itemize}
\item[(1)] We have
\[
L(\Delta;e_0;j)=\begin{cases}
\Upsilon_i, &\quad\text{ if $j=q-1+i(q-1)q$ with $i\in{\Z}$},\\
\Xi_i, &\quad\text{ if $j=i(q-1)q$ with $i\in{\Z}$},\\
0, &\quad\text{ if $j\ne q-1+i(q-1)q, i(q-1)q$ with $i\in{\Z}$}.
\end{cases}
\]
\item[(2)] For an integer $m\ge 0$, we have
\begin{align}
L_{\Delta}(m+1)=&\sum\limits_{0\le q-2+i(q-1)q\le m}
(-1)^{m-(q-2+i(q-1)q)}\binom{m}{q-2+i(q-1)q}\Upsilon_i\notag\\
&\qquad+\sum\limits_{0\le -1+i(q-1)q\le m}
(-1)^{m-(-1+i(q-1)q)}\binom{m}{-1+i(q-1)q}\Xi_i\label{e:comp6}
\end{align}
where the two summations are taken over the integers $i$.
\end{itemize}
\end{prop}
\begin{proof}
Conclusion (1) directly follows from the residue formula
(\ref{e:fbtt21}) and the equation (\ref{e:comp5}) above.

To prove the conclusion (2), we use the action of
$\gamma:=\gamma_{-1}=\left(\begin{array}{cc}1&-1\\0&1\end{array}
\right)$ on the edge $e_0$ as in the proof of Lemma \ref{lem:mdd3}
to get $\overline{e_{1,0}}=\gamma\ast e_0$. Since ${\mathcal U}_1
=U(\overline{e_{1,0}})$, we have for an integer $m\ge 0$
\begin{align*}
L_{\Delta}(m+1) &=\int_{{\mathcal U}_1} x^md\mu_{\Delta}(x)
=\int_{U(\gamma\ast e_0)} x^m d\mu_{\Delta}(x)
=\int_{U(e_0)}(x-1)^md\mu_{\Delta}(x)\\
&=\sum\limits_{j=0}^m(-1)^{m-j}\binom{m}{j}\int_{U(e_0)}x^jd\mu_{\Delta}(x)\\
&=\sum\limits_{j=0}^m(-1)^{m-j}\binom{m}{j}L_{\Delta}(j+1)\\
&=\sum\limits_{0\le q-2+i(q-1)q\le m}
(-1)^{m-(q-2+i(q-1)q)}\binom{m}{q-2+i(q-1)q}\Upsilon_i\\
&\qquad\qquad+\sum\limits_{0\le -1+i(q-1)q\le m}
(-1)^{m-(-1+i(q-1)q)}\binom{m}{-1+i(q-1)q}\Xi_i.
\end{align*}
\end{proof}

\begin{rem}\label{rem:comp1}
We have some remarks related to the values $L_{\Delta}(j)$ below.
\begin{itemize}
\item[(1)] Other than $\Upsilon_0\ne 0$ and $\Xi_1=0$ which we prove in Section \ref{mdd},
we don't know whether the other elements $\Upsilon_i,\,
\Xi_i\in{\bk_\infty}$ vanish or not.
\item[(2)] We don't know whether the elements $\Upsilon_i$ and
$\Xi_i$ are transcendental over $\bk$ (if they are not equal to
$0$).
\item[(3)] The values $L_{\Delta}(j)$ for $1\le j\le q^2-2$ are
calculated in Example \ref{example:comp1} in the following.
\end{itemize}
\end{rem}
We'll state the Lucas' formula of binomial numbers in characteristic $p$, which is
useful in the computation in our cases. Denote by
\[
((i_1,i_2,\cdots,i_s))=\frac{(i_1+i_2+\cdots+i_s)!}{i_1!\:
i_2!\:\cdots\:i_s!}
\]
for any integers $i_1,i_2,\cdots,i_s\geq 0$. We
have the following assertion about the multinomial
numbers by Lucas \cite{Lu}:
\begin{prop}[Lucas]\label{prop:comp2}
For non-negative integers $n_0,n_1,\cdots,n_s$,
\begin{equation}\label{e:comp7}
((n_0,n_1,\cdots,n_s)) \equiv \prod_{j\geq 0}
  ((n_{0,j},n_{1,j},\cdots,n_{s,j})) \mod p
\end{equation}
where $n_i=\sum\limits_{j\geq 0}n_{i,j}\, q^j$ is the
$q$-digit expansion for $i=0,1,\cdots,s$.
\end{prop}
\begin{rem}\label{rem:comp2}
Proposition~\ref{prop:comp2} is useful when $s=1$.
In this case formula (\ref{e:comp7}) is
expressed in the form: let $n=\sum\limits_j n_j\, q^j$ and
$k=\sum\limits_j k_j\, q^j$ be $q$-digit expansion for
non-negative integers $n$ and $k$, then
\begin{equation}\label{e:comp8}
\binom nk \equiv \prod_{j\geq 0} \binom{n_j}{k_j} \mod p.
\end{equation}
\end{rem}

\begin{example}\label{example:comp1}
The values $L_{\Delta}(j)$ for $1\le j\le q^2-2$ can be calculated quickly by putting
in $m=j-1$ in the formula (\ref{e:comp6}), where the index $i=0$
in the first summation and $i=1$ in the second summation
on the right side of the formula. As $\Xi_1=0$, we get
\begin{equation*}
L_{\Delta}(j)=(-1)^{j-q+1}\binom{j-1}{q-2}\Upsilon_0, \quad 1\le j\le q^2-2.
\end{equation*}
After applying Lucas' formula (\ref{e:comp8}), we see that $L_{\Delta}(j)\ne 0$
if and only if $j=q-1+lq, q+lq$ for $l=0, 1, \cdots, q-2$, and
\begin{equation*}
L_{\Delta}(q-1+lq)=(-1)^l\Upsilon_0, \qquad L_{\Delta}(q+lq)=(-1)^l\Upsilon_0.
\end{equation*}
Therefore the functional equation (\ref{e:compfeq}) for $L_{\Delta}(j)$ with
$1\le j\le q^2-2$ essentially becomes
\begin{align*}
L_{\Delta}(q-1+lq)&=(-1)^l\Upsilon_0=(-1)\cdot(-1)^{q-l-2}\Upsilon_0
=-L_{\Delta}((q-l-1)q)\\
&=-L_{\Delta}(q^2-1-(q-1+lq)).
\end{align*}
\end{example}

\vspace{12pt} Due to Gekeler's work ((2) of Theorem \ref{thm:3.1}), we have
$\Delta(z)=-(P_{q+1,1}(z))^{q-1}$, where the Poincar\'e series
$P_{q+1,1}(z)$ is a cusp form of $\Gamma$ of weight $q+1$ and type
$1 \mod (q-1)$. For a cusp form $f$ of $\Gamma$ with weight $n$ and
type $m$, the equation (\ref{e:fbtt5}) becomes
\begin{equation}\label{e:comp10}
\int_{U(\gamma\star e)}g(x)d\mu_f(x)=\int_{U(e)}
(\det(\gamma))^{1-m}(cx+d)^{n-2}g(\gamma x)d\mu_f(x)
\end{equation}
where the action ``$\star$" of $\Gamma$ on $\mathcal T$ is actually ``$\ast$",
see \cite{Go1992} and \cite{Te1992} for detailed exposition on the above equation.
Let $\mathcal{P}(z)$ denote by the Poincar\'e series $P_{q+1,1}(z)$
and $\mu_{\mathcal P}$ the associated measure on
${\P1}(\bk_\infty)$. And let
\begin{equation*}
{\mathscr X}_0=\int_{U(e_0)}d\mu_{\mathcal P}.
\end{equation*}

\begin{prop}\label{prop:comp3}
The measure $\mu_{\mathcal P}$ on ${\P1}(\bk_\infty)$ is $h$-admissible with
$h$ being the smallest integer greater than or equal to $(q-1)/2$, and is
completely determined by the following values:
\begin{align*}
&{\rm (1)\,} \int_{U(e_0)} x^jd\mu_{\mathcal P}(x)=\begin{cases}0,\, &\text{ if \, $0< j\le q-1$},\\
{\mathscr X}_0\ne 0,\, &\text{ if  $j=0$}.\end{cases}\qquad\qquad\qquad\qquad\\
&{\rm (2)\,} \int_{U(e_n)}x^jd\mu_{\mathcal P}(x)=0,\, \text{ for
\,$0\le j\le q-1$ and any $n\ge 1$.}
\end{align*}
The notations $e_n$ for $n\in{\Z}$ are the same as those in Section \ref{mdd}.
\end{prop}
\begin{proof}
We'll only show that the equation (1) of the
theorem holds here, the rest of the proof is the same as those of
Lemma \ref{lem:mdd3}, Theorem \ref{thm:mdd2}, and Corollary \ref{cor:mdd1}.

The action of $\delta=\left(\begin{array}{cc}0&1\\1&0\end{array}\right)$
on $e_0$ gives $\overline{e_{-1}}=\delta\ast e_0$. Therefore
\begin{equation*}
\int_{U(e_{-1})}x^jd\mu_{\mathcal P}(x)=-\int_{U(\delta\ast e_0)}
x^jd\mu_{\mathcal P}(x)=-\int_{U(e_0)}x^{q-1-j}d\mu_{\mathcal P}(x),
\end{equation*}
where in the second equality we need to apply the equation (\ref{e:comp10})
since the type of the Poincar\'e series ${\mathcal P}(z)$ is $1 \mod (q-1)$.
Then from
\begin{align*}
\sum\limits_{b\in{\Fq}^{*}}\int_{U(\gamma_b\ast e_0)}
x^jd\mu_{\mathcal P}(x)+\int_{U(e_0)}x^jd\mu_{\mathcal P}(x)
-\int_{U(e_{-1})}x^jd\mu_{\mathcal P}(x)=0
\end{align*}
we get
\begin{align}
-\int_{U(e_0)}x^{q-1-j}d\mu_{\mathcal P}(x)
&=\int_{U(e_0)}x^jd\mu_{\mathcal P}(x)+\sum\limits_{b\in{\Fq}^*}
  \int_{U(e_0)}(x+b)^jd\mu_{\mathcal P}(x)\notag\\
&=\int_{U(e_0)}x^jd\mu_{\mathcal P}(x)+\sum\limits_{i=0}^j
  \binom{j}{i}\sum\limits_{b\in{\Fq}^*}b^i \int_{U(e_0)}x^{j-i}d\mu_{\mathcal P}(x)
  \label{e:comp11}
\end{align}
As $0\le i\le q-1$, we have the equality
\begin{equation*}
\sum_{b\in{\Fq}^*} b^i=
\begin{cases}
0, \quad &\text{ if $0< i< q-1$},\\
-1, \quad &\text{ if $i=0, \, q-1$},
\end{cases}
\end{equation*}
thus, the equation (\ref{e:comp11}) can be written as
\begin{equation*}
-\int_{U(e_0)}x^{q-1-j}d\mu_{\mathcal P}(x)=
\begin{cases}
0, &\quad\text{ if $0\le j<q-1$},\\
-\int_{U(e_0)}d\mu_{\mathcal P}(x), &\quad\text{ if $j=q-1$}.
\end{cases}
\end{equation*}
Therefore we get (1) of the theorem except for ${\mathscr X}_0\ne 0$. But this
conclusion can be proved in the same way as Corollary \ref{cor:mdd1}.
\end{proof}

We denote an element $\gamma\in {\GL2}(\A)$ by
$\gamma=\left(\begin{array}{cc}a_{\gamma}&b_{\gamma}\\
  c_{\gamma}&d_{\gamma}\end{array}\right) $, then get
\begin{cor}\label{cor:comp1}
For $0\le j\le q-1$,
\begin{equation*}
\int_{U(e)}x^jd\mu_{\mathcal P}(x)=
\begin{cases}
0,\; & \text{ if $e=\gamma\ast e_n$ for some $n\ge 1$ and some
$\gamma\in{\Gamma}$},\\
b_{\gamma}^j d_{\gamma}^{q-1-j}{\mathscr X}_0, \; &
\text{ if $e=\gamma\ast e_0$ for some $\gamma\in\Gamma$}.
\end{cases}
\end{equation*}
\end{cor}

\begin{rem}\label{rem:comp3}
Although the Drinfeld discriminant $\Delta(z)$ and the Poincar\'e series
${\mathcal P}(z)$ are related by $\Delta(z)=-({\mathcal P}(z))^{q-1}$, we don't know
if there are direct relations in the space of $h$-admissible measures ($h$ big enough)
between their associated measures $\mu_{\Delta}$
and $\mu_{\mathcal P}$, or even there are direct relations between the
constants $\Upsilon_0$ and ${\mathscr X}_0$ as elements of ${\C}_\infty$.
In the general case, for any $f\in S_n(\Gamma)$, Teitelbaum \cite{Te1991} has proved
that $\int_{U(e_i)} x^jd\mu_f(x)=0$ for $0\le j\le n-2$ and all $i\ge i_0$ for some
integer $i_0$, but how much extension we can say about the values
$\int_{U(e_i)} x^jd\mu_f(x)$ for $0\le j\le n-2$ and $0\le i< i_0$ and how they are
related to the special values of the characteristic $p$ valued ``$L$-function"
$L_f(s)$ are not clear.
\end{rem}

\begin{rem}\label{rem:comp4}
The zeta function $\zeta(s)$ for $s=(x,y)\in S_{\infty}={\C}_\infty^{\ast}\times
{\Z}_p$ in Section \ref{introduction} has special values $\zeta(x, -j)$ for
integers $j\ge 0$. And $\zeta(T^j, j)={\rm z}(j)$ $=\sum_{a\in {\A}^{+}}
\frac{1}{a^j}$ as mentioned in the beginning of Section \ref{introduction}.
The zeta measure $\mu^{(\infty)}_x=x\nu^{(\infty)}_x$ is a measure on
${\A}_\infty$ given by (see Section $3$ of \cite{Ya2001})
\[
\nu^{(\infty)}_x=\sum\limits_{k=1}^\infty (-1)^k x^{-k}
{G}^*_{q^k-1}
\]
where the measures $\{{G}^*_{q^k-1}\}_{k\ge 1}$ are given as follows: let
${\A}_\infty=\bigsqcup_i B_{\alpha_i}(|\pi|^l)$ be a disjoint decomposition into
closed balls, where $l$ is chosen such that $l\ge k$, $\alpha_i\in
{\Fq}[\pi]$, and $\deg_{\pi}(\alpha_i)<l$ for each $i$. Each measure
${G}^*_{q^k-1}$ is $0$-admissible, and is given by (see the computation in
Section $3$ of \cite{Ya2001})
\begin{equation}\label{e:comp12}
{G}^*_{q^k-1}(B_{\alpha_i}(|\pi|^l))=
\begin{cases}
(-1)^k, & \text{ if $\deg_{\pi}(\alpha_i)<k$ },\\
0,   & \text{ if $k\le \deg_{\pi}(\alpha_i)<l$ }.
\end{cases}
\end{equation}
By using Example \ref{example:fbtt3}, we see that the measures
$\{{G}^*_{q^k-1}\}_{k\ge 1}$ can be easily expressed as functions
on the subtree $\mathcal W$ of $\mathcal T$, where $\mathcal W$ is obtained
from $\mathcal T$ by cutting
every vertex and every edge in the paths starting at the vertex $\Lambda_1$
except for those containing the edge $\Lambda_1\Lambda_0$:
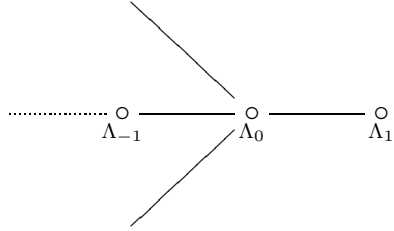
\begin{equation*}
\xymatrix@=36pt{
                 &\ar@{}[dr] &  &\\
&{\circ}\ar@{-}[r]\ar@{ }_{\Lambda_{-1}}\ar@{.}[l] &{\circ}\ar@{-}[ul]\ar@{-}[r]
\ar@{}_{\Lambda_0} &{\circ}
\ar@{}_{\Lambda_1} & \\
                          &\ar@{-}[ur] & &
}
\end{equation*}
There are many ways to extend such functions to
harmonic functions on the (oriented) edges of $\mathcal T$.
But unlike the measures associated to cusp forms, the measures
${G}^*_{q^k-1}$, $\nu^{(\infty)}_x$, and
$\mu^{(\infty)}_x$ clearly lack symmetries under the action of $\Gamma$.
Further study is necessary in order to understand them better.
\end{rem}

\vspace{12pt}
In Section \ref{cpvm}, we have discussed the characteristic $p$ valued distributions
by using $C^h$ functions. Although $C^h$ functions and the functions with Lipschitz
conditions are not studied very often in rigid analytic geometry and
related topics, the study of their dual spaces (the spaces of
$h$-admissible measures and their variants) may have interesting
applications in the theory of ergodic functions over ${\Fq}[[\pi]]$. In the
case $q=2$, an ergodic function $f: {\F}_2[[\pi]]\to {\F}_2[[\pi]]$ is a continous
function
\[
f(x)=\sum\limits_{j=0}^\infty a_j\, G_j(x)
\]
(where $G_j(x)$ for $j\ge 0$ are the Carlitz polynomials given in Section \ref{cpvm})
which satisfies the following conditions on the coefficients $a_j$ for $j\ge 0$
(see \cite{LSY2011}) :
\begin{itemize}
\item[(1)] $a_{0}\equiv1\mod \pi$,\,\; $a_{1}\equiv 1+\pi \mod
{\pi}^{2}$, \,\; $a_{3}\equiv {\pi}^{2} \mod {\pi}^{3}$;
\item[(2)] $\left\vert a_{j}\right\vert <|\pi|^{[\log_2(j)]}
=2^{-[\log_2 (j)]}$, for $j\geq 2$;
\item[(3)] $a_{2^{j}-1}\equiv {\pi}^{j} \mod {\pi}^{j+1}$ for $j\ge 2$.
\end{itemize}
We can see from the above descriptions and (2) of Theorem \ref{thm:cpvm3} that ergodic
functions on ${\F}_2[[\pi]]$ are not $C^1$ functions
(therefore certainly are not locally analytic functions), but are continuous
functions satisfying the so-called $1$-Lipschitz condition (see Chapter 3
of \cite{AK2009} or \cite{LSY2011} for the concept ``$1$-Lipschitz" condition).
In the case $q=2$, the measures $\mu_{\Delta}$ and $\mu_{\mathcal P}$
are $1$-admissible. After checking Teitelbaum's estimation (\ref{e:fbtt6}) and
the construction of integrals in Lemma \ref{lemma:cpvm4} more carefully, we can
see that the ergodic functions on ${\F}_2[[\pi]]$ are integrable against
$\mu_{\Delta}$ or $\mu_{\mathcal P}$. Due to the applications of the
theory of ergodic functions in non-Archimedean analysis to cryptography, we
wish that further studies of the duality between functions and measures
of characteristic $p$ would be helpful in applications.

\end{document}